\documentclass[12pt]{amsart}

\usepackage{amsmath}
\usepackage{caption}
\usepackage[curve]{xypic}
\usepackage{cite}
\usepackage{enumitem}
\usepackage{color}
\usepackage{url}
\usepackage[usenames,dvipsnames]{xcolor}
\usepackage{graphicx}
\usepackage{tikz-cd}
\usepackage{hyperref}

\usepackage{caption}
\usepackage{subcaption}

\makeatletter 
\def\@cite#1#2{{\m@th\upshape\bfseries%
[{#1\if@tempswa{\m@th\upshape\mdseries, #2}\fi}]}}
\makeatother 

\theoremstyle{plain}
\newtheorem{thm}{Theorem}[section]
\newtheorem{cor}[thm]{Corollary}

\newtheorem{lem}[thm]{Lemma}

\theoremstyle{definition}
\newtheorem{defn}[thm]{Definition}

\newtheorem{exercise}[thm]{Exercise}

\theoremstyle{remark}
\newtheorem{rem}[thm]{Remark}

\numberwithin{equation}{subsection}
\renewcommand{\bold}[1]{\medskip \noindent {\bf #1 }\nopagebreak}

\newcommand{\nc}{\newcommand}
\newcommand{\rnc}{\renewcommand}
\newcommand{\e}{\varepsilon}

\nc\bA{\mathbb{A}}
\nc\bB{\mathbb{B}}
\nc\bC{\mathbb{C}}
\nc\bD{\mathbb{D}}
\nc\bE{\mathbb{E}}
\nc\bF{\mathbb{F}}
\nc\bG{\mathbb{G}}
\nc\bH{\mathbb{H}}
\nc\bI{\mathbb{I}}
\nc{\bJ}{\mathbb{J}} 
\nc\bK{\mathbb{K}}
\nc\bL{\mathbb{L}}
\nc\bM{\mathbb{M}}
\nc\bN{\mathbb{N}}
\nc\bO{\mathbb{O}}
\nc\bP{\mathbb{P}}
\nc\bQ{\mathbb{Q}}
\nc\bR{\mathbb{R}}
\nc\bS{\mathbb{S}}
\nc\bT{\mathbb{T}}
\nc\bU{\mathbb{U}}
\nc\bV{\mathbb{V}}
\nc\bW{\mathbb{W}}
\nc\bY{\mathbb{Y}}
\nc\bX{\mathbb{X}}
\nc\bZ{\mathbb{Z}}
\nc\cA{\mathcal{A}}
\nc\cB{\mathcal{B}}
\nc\cC{\mathcal{C}}
\rnc\cD{\mathcal{D}}
\nc\cE{\mathcal{E}}
\nc\cF{\mathcal{F}}
\nc\cG{\mathcal{G}}
\rnc\cH{\mathcal{H}}
\nc\cI{\mathcal{I}}
\nc{\cJ}{\mathcal{J}} 
\nc\cK{\mathcal{K}}
\rnc\cL{\mathcal{L}}
\nc\cM{\mathcal{M}}
\nc\cN{\mathcal{N}}
\nc\cO{\mathcal{O}}
\nc\cP{\mathcal{P}}
\nc\cQ{\mathcal{Q}}
\rnc\cR{\mathcal{R}}
\nc\cS{\mathcal{S}}
\nc\cT{\mathcal{T}}
\nc\cU{\mathcal{U}}
\nc\cV{\mathcal{V}}
\nc\cW{\mathcal{W}}
\nc\cY{\mathcal{Y}}
\nc\cX{\mathcal{X}}
\nc\cZ{\mathcal{Z}}

\nc{\dmo}{\DeclareMathOperator}
\rnc{\Re}{\operatorname{Re}}
\rnc{\Im}{\operatorname{Im}}
\rnc{\span}{\operatorname{span}}
\dmo{\rank}{rank}
\dmo{\End}{End}
\dmo{\Hom}{Hom}
\dmo{\Jac}{Jac}
\dmo{\Id}{Id}
\dmo{\Ann}{Ann}
\dmo{\Area}{Area}
\dmo{\CP}{\bC P^1}
\dmo{\vol}{Vol}
\dmo{\ML}{\mathcal{ML}}
\dmo{\MF}{\mathcal{MF}}
\nc{\Shear}{\operatorname{Shear}}

\title{Mirzakhani's work on earthquake flow}
\author[Wright]{Alex~Wright}
\begin{document}
\maketitle
\thispagestyle{empty}

\begin{abstract}
The Teichm\"uller unipotent flow can be defined concretely on certain moduli spaces of singular flat surfaces   by shearing polygonal presentations of the surfaces. Thurston's earthquake flow on moduli spaces of hyperbolic surfaces is  more mysterious. Both flows have deep and important connections to other areas of mathematics. 

In this expository survey we give a geometric account of the main ideas behind Mirzakhani's  theorem relating these two flows. Our presentation avoids some technical prerequisites that featured in the original more analytic presentation.
\end{abstract}

\setcounter{tocdepth}{1} 
\tableofcontents



\section{Introduction}

\bold{Flat geometry.} One entry point to Mirzakhani's work is the $SL(2,\bR)$ action on moduli spaces of Abelian or quadratic differentials. Due to its close connection to Teichm\"uller theory, applications to topics such as rational billiards, and deep analogy to Lie group actions on homogeneous spaces, this action has been extensively studied; see for example the author's short survey \cite{FBTMS} for an introduction. 

Part of the appeal of this action is that it is easily defined. Abelian and quadratic differentials can be presented via polygons in $\bR^2$, with edges identified in pairs to define a surface with a singular flat metric. The $SL(2,\bR)$ action arises quite simply from the usual $SL(2,\bR)$ action on $\bR^2$, which linearly transforms one polygonal presentation into another.

\bold{Hyperbolic geometry.} The other entry point is the earthquake flow defined on the bundle of measured laminations over a moduli space of hyperbolic surfaces. This flow is only easy to visualize when the measured lamination is a closed curve, in which case the result of the earthquake is obtained by cutting the geodesic representative of the curve and re-gluing with a twist. More generally, the measured lamination represents the ``fault lines" along which a ``sliding" or ``shearing" deformation occurs. When the lamination has a fractal structure, infinitesimal shearing along each individual ``fault line" can combine to a definite change to the surface. 

Although the definition of earthquake flow is less elementary, its applications and connections are far reaching. Early in its history it was used prominently by Kerckhoff to resolve the longstanding Nielsen Realization Problem about  mapping class groups \cite{K}. More recently, Mirzakhani and others used it to prove equidistribution results in the moduli space of hyperbolic surfaces \cite{RHS, AR0}, which are a crucial tool for example in one of Mirzakhani's results on counting geodesics on individual surfaces  \cite{mirzakhani2016counting}.

\bold{Mirzakhani's bridge between hyperbolic and flat geometry.} Although defined on different spaces and featuring different notions of geometry, the two flows have similarities. Both involve some notion of shearing, and both have analogies to unipotent flows on homogeneous spaces. Both have strong non-divergence properties, established simultaneously in a single paper by  Minsky and Weiss  \cite{MW2}. As discussed in Section \ref{S:Ham}, both are Hamiltonian with respect to natural functions and natural symplectic forms. 

The purpose of this survey is to give an expository introduction to Mirzakhani's discovery that, in an abstract measurable sense, both flows are the actually the same. This is the main result of \cite{M}.

\begin{thm}\label{T:main}
There is a measurable mapping class group equivariant conjugacy $F$ between the earthquake flow $(\lambda, X)\mapsto (\lambda, E_{t\lambda}(X))$ on the bundle $\ML\times \cT_g$ of measured laminations over Teichm\"uller space $\cT_g$ and 
the Teichm\"uller unipotent flow action of 
$$u_t=\left(\begin{array}{cc} 1& t\\ 0 & 1\end{array}\right)$$
on the bundle $QD$ of nonzero quadratic differentials  over Teichm\"uller space.
\end{thm}

That $F$ is a conjugacy means that the following diagram commutes. 
\[
\begin{tikzcd}
\ML\times \cT_g \arrow{r}{E_t} \arrow[swap]{d}{F} & \ML\times \cT_g \arrow{d}{F} \\
QD  \arrow{r}{u_t} & QD
\end{tikzcd}
\]
We will later discuss the natural Lebesgue class measure on $\ML\times \cT_g$ and see that it is the pull back of Masur-Veech measure, but for the moment it suffices to understand that $F$ is Borel-measurable but {not} continuous.

\bold{Significance and applications.} Theorem \ref{T:main}  builds a surprising bridge between the more mysterious world of earthquake flow and the comparatively well understood Teichm\"uller unipotent flow. Perhaps the most important consequence is the following. 

\begin{cor}
Earthquake flow is ergodic. 
\end{cor}

\begin{proof}
It is well known that Teichm\"uller unipotent flow is ergodic: This follows from from the Howe-Moore Theorem and the ergodicity of Teichm\"uller geodesic flow. (See for example the textbook \cite[Section III]{BM} for an introduction to the Howe-Moore Theorem, which applies here because both the Teichm\"uller geodesic and unipotent flows are part of the $SL(2, \bR)$ action. See for example the survey \cite[Section 4]{FM} for the ergodicity of Teichm\"uller geodesic flow, which was originally proved independently by Masur and Veech.) 
\end{proof}

This ergodicity is a key ingredient in the equidistribution applications mentioned above. 

As McMullen describes in his laudation for Mirzakhani's field medal, Theorem \ref{T:main} can be viewed not only as a bridge between flat and hyperbolic geometry, but also a bridge across the ``holomorphic/symplectic" divide \cite{Laudation}. Indeed, quadratic differentials live primarily in the holomorphic world, and since the only concise definition of earthquake flow is as a Hamiltonian flow it is natural to think of earthquakes as living more in the symplectic world. 
 
\bold{Origin and intended audience.} This survey is aimed at the many mathematicians working in dynamics or low dimensional topology who may have heard the statement of Theorem \ref{T:main} but are completely unfamiliar with its proof. We include a limited introduction to some of the necessary background material, in hopes that this will make our exposition readable to a junior graduate student who is able to get some assistance from their advisor or from other sources. The reader should have some familiarity with quadratic differentials, but we have included enough exposition on earthquake flow that the reader who is willing to take some things on faith should be able to gain a meaningful level of understanding of the proof even with no previous exposure to earthquake flow. 


  Our goal is to present the main proof in the most  elementary and geometric way possible. Only after accomplishing that will we proceed to discuss the more sophisticated results that give additional understanding and perspective.  Most of the background material we present is due to people other than Mirzakhani, especially  Thurston and Bonahon. However all the material is chosen to allow us to prove and appreciate Mirzakhani's result. 
  
  Our presentation of the proof of Theorem \ref{T:main} is not exactly the same as Mirzakhani's, in that here we do not make use of transverse cocycles or the symplectomorphism of Bonahon-Sozen \cite{BS}. See Remark \ref{R:differences}.

These notes were originally written to  accompany lectures at the 2018  summer school on Teichm\"uller dynamics, mapping class groups and applications at Grenoble, as well as lectures at the 2018 summer school on Teichm\"uller Theory and its Connections to Geometry, Topology and Dynamics at the Fields Institute in Toronto. 

\bold{Updates.} Since this survey was written in 2018, the author also wrote a much comprehensive survey on Mirzakhani's work aimed at a broader audience \cite{Tour}. Additionally, this survey has motivated further work, and the questions and conjectures proposed in Remarks \ref{R:continuous} and \ref{R:strat} have now been completely resolved \cite{NoConj,calderon2021shear}.  

\bold{Acknowledgments.} I am happy to thank Francisco Arana-Herrera, Francis Bonahon, Aaron Calderon, Steve Kerckhoff, Jeremy Kahn, and Kasra Rafi for helpful conversations. I am also grateful to thank Francisco Arana-Herrera, Dat Nguyen, Weston Ungemach, and Adva Wolf for attending and offering very helpful feedback on a test run of the lectures at Stanford the week before Grenoble, and to Yueqiao Wu for pointing out some corrections. 

Some of the figures were created using  \cite{Pi}. I thank 
Yen Duong,
Aaron Fenyes,
Subhojoy Gupta,
Bruno Martelli,
Athanase Papadopoulos,
Guillaume Th\'eret, and
Mike Wolf
for permission to reproduce figures from other sources.

\section{Preliminaries}

In this section we offer a sketchy introduction to the main objects in Theorem \ref{T:main}. It is included mainly for students coming from flat geometry, who may have seen measured foliations but may benefit from an overview of their relation to laminations and an introduction to earthquakes. Experts should certainly skip this section. 

There are many sources for this material, although we have not found one that presents all the material we need in a way which is both rigorous and quickly understood. In addition to the references cited below, learners might be interested in consulting books like \cite{Cal, FarbM, FLP, PennHar, ThurstonNotes}.  

\bold{Foliations and laminations.} In these notes we consider only closed surfaces of genus $g$ at least 2. A measured foliation is a foliation with finitely many prong type singularities, with a transverse measure. This measure assigns a non-negative number to each transverse arc in such a way that arcs isotopic through transverse arcs with endpoints on the same leaves have the same measure.  A saddle connection of a measured foliation is an arc of the foliation joining two singularities.

Measured foliations are typically considered to be equivalent if they differ via isotopy and Whitehead moves, which are moves that collapse saddle connections to split a higher order prong singularity into lower order singularities joined by a saddle connection, as in Figure \ref{F:Whitehead}. 

\begin{figure}[h!]
\includegraphics[width=.6\linewidth]{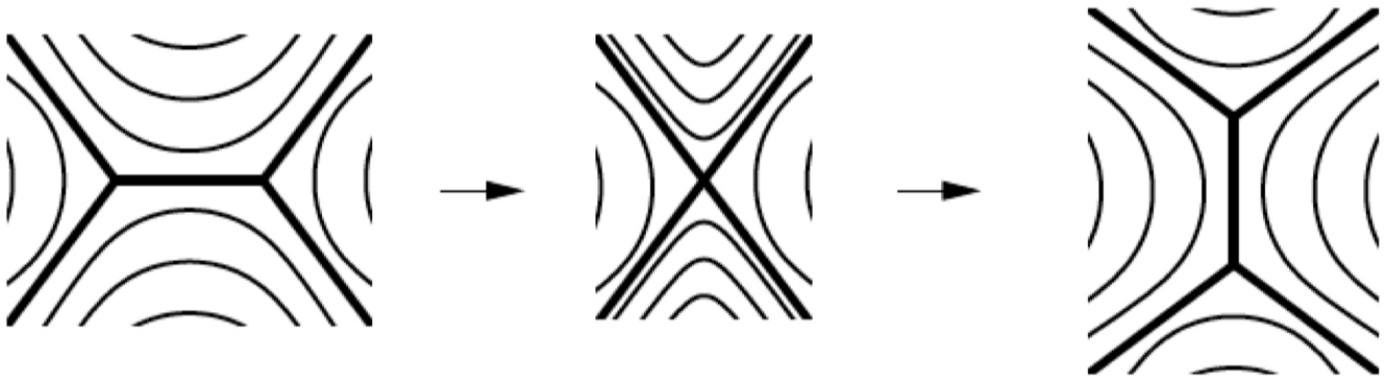}
\caption{Whitehead moves. Picture from \cite{GW}. }
\label{F:Whitehead}
\end{figure}

\begin{rem}
The typical measured foliation has only 3-pronged singularities and no saddle connections, and hence does not admit any Whitehead moves. 
\end{rem}

The space of measured foliations up to Whitehead moves and isotopy is denoted $\MF$. 

\begin{rem}
A celebrated result of Thurston is that $\MF$ is homeomorphic to $\bR^{6g-6}$. We will not make use of this fact. 
\end{rem}

A measured geodesic lamination is a closed subset of a hyperbolic surface  foliated by non-intersecting geodesics with a transverse measure of full support. See \cite[Section 11.6]{Ka} and \cite[Section 8.3]{M} for the basic properties of measured geodesic laminations. 

\begin{rem}
A closed multi-curve is an example of a measured geodesic lamination. However if you take a ``typical" geodesic lamination and intersect it with a transverse arc, you will get a Cantor set with a non-atomic measure. 
\end{rem}

\begin{figure}[h!]
\includegraphics[width=.6\linewidth]{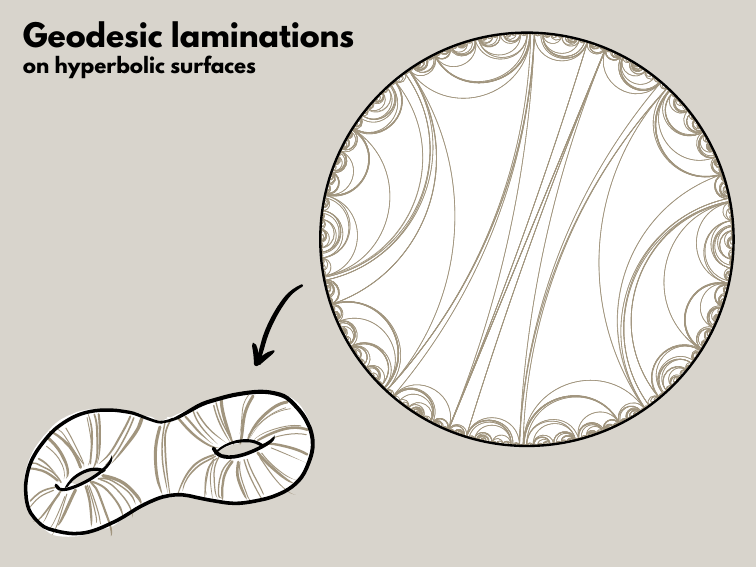}
\caption{A geodesic lamination. Picture from \cite{D}, created by Aaron Fenyes. A similar figure appears in \cite{Fenyes, Fthesis}.}
\label{F:lamination}
\end{figure}

If $\lambda$ is a geodesic lamination on $X$, the  connected components of $X\setminus \lambda$ are called the complementary regions. There are finitely many. Each is bounded by geodesics. The complementary regions can be ideal polygons, and can also be surfaces with genus that are bounded by  closed geodesics and/or ``crowns" of geodesics meeting in cusps.  

\begin{figure}[h!]
\includegraphics[width=.25\linewidth]{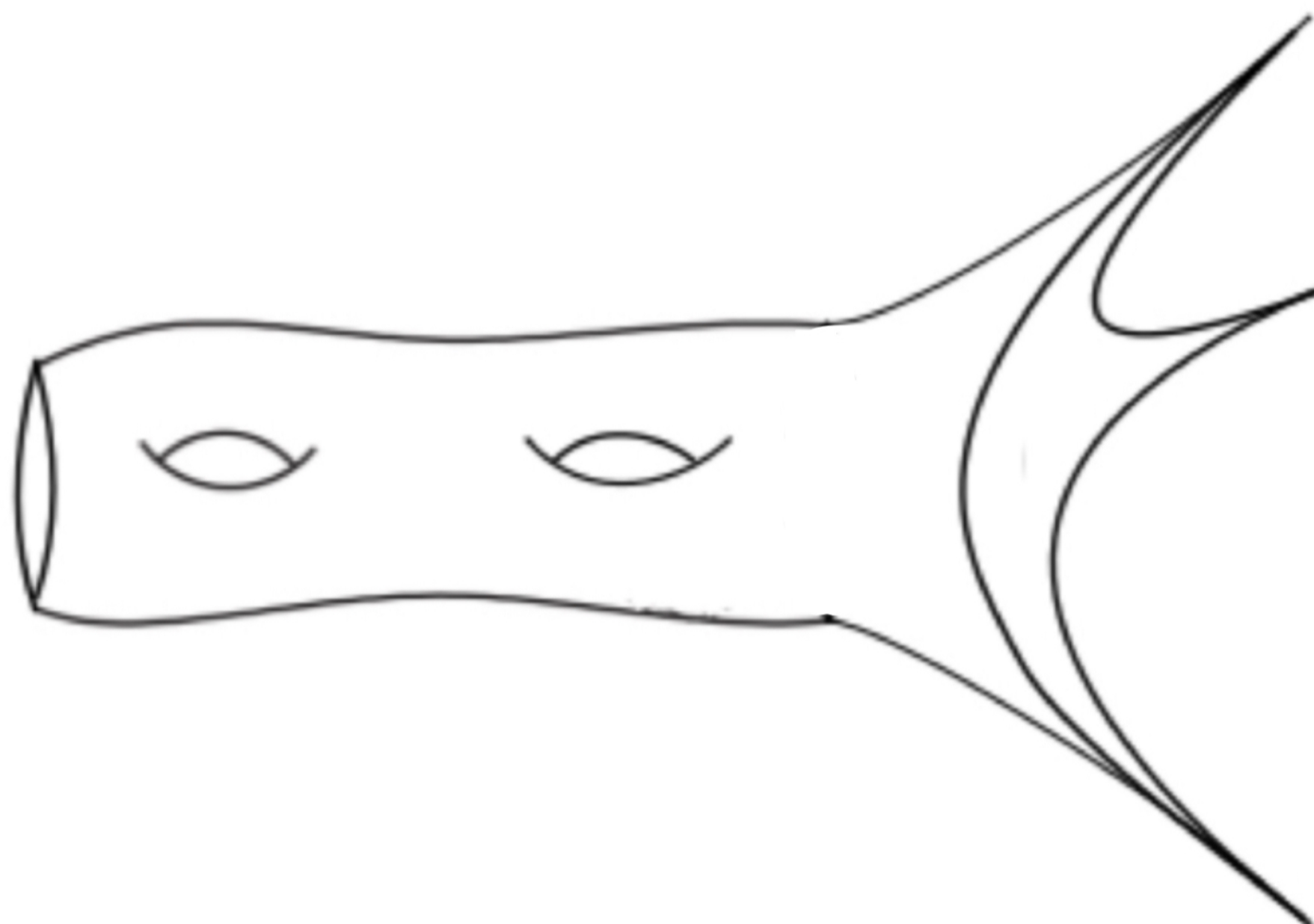}
\caption{A possibly complementary region bounded by a closed geodesic and a ``crown". Picture adapted from \cite{Gu}. }
\label{F:lamination}
\end{figure}

The universal cover of $X$ can be identified with the hyperbolic plane $\bH$. 
Geodesics in $\bH$ correspond to unordered pairs of distinct points on the circle $S^1$ at infinity. Given two different points $X,Y$ of Teichm\"uller space, one obtains an isotopy class of maps from $X$ to $Y$. Lift such a map to obtain a map from $\bH=\tilde{X}$ to $\bH=\tilde{Y}$. 

One can show that any such map from $X$ to $Y$ is a quasi-isometry, so the lifted map from $\bH=\tilde{X}$ to $\bH=\tilde{Y}$  extends to a homeomorphism between the circles at infinity. Hence geodesics on $X$ are in correspondence with geodesics on $Y$, by considering the endpoints of the geodesic on $S^1$. In this way a measured geodesic lamination for one hyperbolic metric uniquely determines one for any other hyperbolic metric, and we can think of measured geodesic laminations as topological rather than metric objects. Denote the set of all measured geodesic laminations by $\ML$. 

We define a line of a measured foliation to be either a leaf not passing through a singularity, or any leaf that is a limit of non-singular leaves. Note that if a line passes through a singularity, it enters and exits the singularity on adjacent prongs. (Those inclined to think about quadratic differentials can think of this as having angle $\pi$ at every singularity.)  

\begin{lem}\label{L:Cornered}
Every line of a measured foliation also determines a pair of distinct points in $S^1$. 
\end{lem}

\begin{proof}[Cartoon of the proof.]
Consider a simple closed curve that the leaf passes through infinitely many times but with no unnecessary intersections that could be removed by an isotopy. The leaf gets ``cornered" by  lifts of the simple curve to smaller and smaller regions of $\bH$, as seen from a fixed basepoint, forcing the leaf to converge to the intersection of these half-spaces, as in Figure \ref{F:Cornered}. 
\begin{figure}[h!]
\includegraphics[width=.25\linewidth]{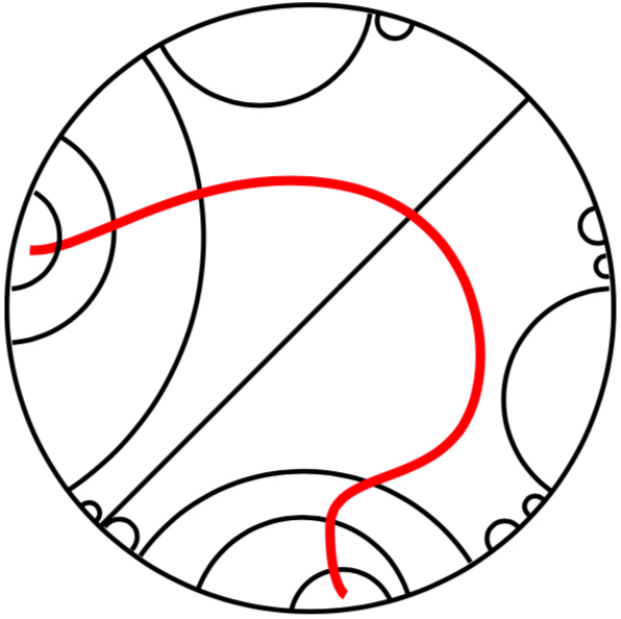}
\caption{The proof of Lemma \ref{L:Cornered}. Picture from \cite[Figure 8.12]{Mar}.}
\label{F:Cornered}
\end{figure}
\end{proof}

\begin{rem}
It is not so easy to prove the desired simple closed curve exists. One possibility is to use a ``normal form" for the foliation \cite[Section 6.4]{FLP}.
\end{rem}

\begin{rem}
If the measured foliation is known to arise from a quadratic differential, one can alternatively use the fact that the hyperbolic and flat metrics are quasi-isometric \cite[Section 5.3]{Ka}.  
\end{rem}

By replacing each line in a measured foliation by the  geodesic with the same endpoints, one obtains an associated measured lamination \cite[p. 251]{Ka}. This procedure ``tightens" each leaf to a geodesic. 

\begin{figure}[h!]
\includegraphics[width=.25\linewidth]{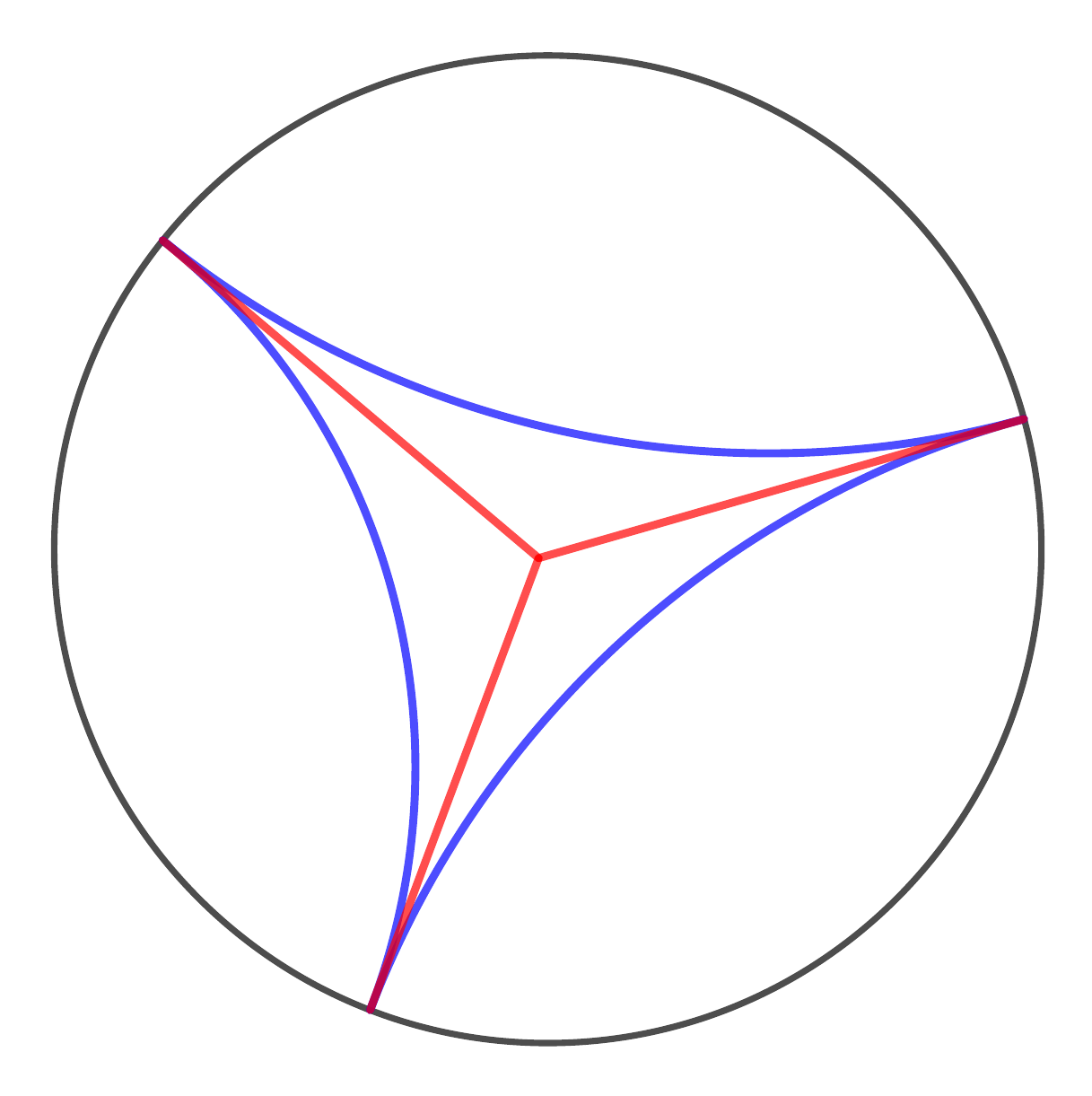}
\caption{Each three pronged singularity that does not lie on any saddle connections gives an ideal triangle in the associated ``tightened" lamination. }
\label{F:Simple}
\end{figure}

\begin{thm}
The tightening map $\MF \to \ML$ is a homeomorphism. 
\end{thm}

I do not know any short proof of this result, but a good reference on the tightening map and related topics is \cite{L}. One approach is to build an inverse map using train tracks, but this involves not only showing that every lamination is carried by a train track, but also that the measured foliations constructed using different choices of train track differ by Whitehead moves. See for example \cite[Section 5]{Leininger} for an expository account.

A measured lamination is called maximal if its complementary regions are all ideal triangles. The value of the inverse map $\ML \to \MF$ on a maximal measured lamination can be visualized by a ``collapsing" procedure. This procedure is perhaps initially somewhat mind bending, but nonetheless we briefly give the flavor of how it is performed. 
One ``collapses" or ``pinches" each triangle onto a ``skeleton" consisting of three lines, one going towards each cusp of the triangle, meeting in a central point. This eliminates all the ``empty space" not covered by the lamination, and the result is a foliation. 

\begin{rem}
One should compare this collapse map to the map $x\mapsto \int_0^x d\mu$ for a measure $\mu$ on a Cantor set in $\bR$. This collapses all the intervals not included in the Cantor set.
\end{rem}

In this way we see that each complementary triangle corresponds to a three-pronged singularity. This can be easily extended to the case when all complementary regions are ideal polygons. One could try to extend it further to the case where the complementary regions have genus by extending the lamination to a maximal lamination, but then one has to show that the  measured foliations resulting from different extensions differ by Whitehead moves, and this is not obvious. The case when the measure gives positive measure to some closed curves requires a special argument, since in this case the foliation can't be obtained by collapsing. 

\bold{Intersection number.} There is a continuous intersection number function  $$i:\ML\times \ML\to \bR_{\geq0},$$ which, restricted to weighted simple multi-curves, is the linear extension of the usual geometric intersection number. One can show that weighted simple multi-curves are dense in $\ML$, so this uniquely determines $i$, however there are also easy direct definitions \cite{Bcurrents}. Since $\MF \simeq \ML$, we also get an intersection number on $\MF$. 

If $\alpha\in \MF$ and $\beta$ is a simple curve, $i(\alpha, \beta)$ is the inf over all ways of realizing $\beta$ as a sequence of arcs transverse to $\alpha$ of the sum of the transverse $\alpha$ measures of these arcs. This can be extended linearly to the case of $\beta$ a simple weighted multi-curve, and by continuity to any $\beta\in \ML$. 

\begin{rem}
The topology on $\MF$ and $\ML$ is the weakest topology for which the function $\lambda\mapsto i(\lambda, \gamma)$ is continuous for each simple closed curve $\gamma$.  
\end{rem}

\bold{Quadratic differentials.} Define $\Delta\subset \MF\times \MF$ by
$$\Delta=\{(\alpha, \beta): i(\alpha, \gamma)=0=i(\beta,\gamma) \text{ for some } \gamma \in \MF\} .$$ 

Note that $\Delta$ contains the diagonal $\{(\alpha, \alpha)\}$ (just take $\gamma=\alpha$), so we can think of $\Delta$ as a ``generalized" or ``fat" diagonal.  

A quadratic differential $q$ determines two measured foliations, namely the horizontal one $h(q)$ and the vertical one $v(q)$. 

\begin{lem} 
For any $q$, $(h(q), v(q)) \notin \Delta$.
\end{lem}
\begin{proof}[Proof sketch]
 Otherwise  take a sequence of weighted simple curves $\gamma_i$ converging to the $\gamma$ showing that $(h(q), v(q))\in \Delta$. Since $i(\gamma_i, h(q))\to 0$, there is a sequence of saddle connections representing $\gamma_i$ whose  sum of absolute values of $x$-components is tiny compared to the total length. (If the holonomy of a saddle connection is $x+iy$, we call $x$ the ``$x$-component" and $y$ the ``$y$-component".) Using the corresponding statement with horizontal and vertical switched, we can obtain a contradiction. (If a curve is represented by a sum of saddle connections whose sum of $x$-components is less than $C$, then the same is true of the flat geodesic representative. As you ``tighten" to get the flat geodesic representative, the $x$ and $y$ coordinates don't get bigger. In fact, both components get monotonically smaller. To make this precise, you need to define an appropriate tightening procedure. Alternatively, see  \cite[Proposition 3.9]{Leininger} for a proof that the flat geodesic gives intersection number with the horizontal and vertical foliations.) 
\end{proof}

Hence we obtain a map from the bundle $QD$ of non-zero quadratic differentials over Teichm\"uller space to $\MF\times \MF\setminus \Delta$ given by $q\mapsto (h(q), v(q))$. 

\begin{rem}
The intersection number $i(h(q), v(q))$ is the area of $q$. 
\end{rem}

\begin{thm}\label{T:QDMF}
 The map $q\mapsto (h(q), v(q))$ determines a homeomorphism $QD\to \MF\times \MF\setminus \Delta.$
 \end{thm}

\begin{proof}[Proof sketch]
One can create an inverse map as follows. Given $$(h,v)\in \MF\times \MF\setminus \Delta,$$ tighten each $h,v$ to geodesic laminations, also denoted $h,v$. Since $(h,v)\notin \Delta$, we have that $h$ and $v$ do not share any leaves, and that each complementary region of $h\cup v$  is a compact polygon:
\begin{itemize} 
\item Indeed, if $h$ and $v$ shared a leaf, then a weak star limit of the Lebesgue measure supported larger and larger segments of this leaf would give a measured foliation $\gamma$ with $i(h, \gamma)=0=i(v, \gamma)$. 
\item If there is a complementary region that isn't a polygon, one could pick a simple curve $\gamma$ in  that region.
\item To see that the polygons are compact, i.e. that none of the vertices are at infinity, requires a bit of extra argument again using weak star limits. 
\end{itemize}

Collapsing all the connected components of the complement of $h\cup v$, as well as all connected components of $h\setminus v$ and $v\setminus h$,  defines a quadratic differential by picking local coordinates $z$ for which $\Re(z)$ and $\Im(z)$ locally coincide with the two foliations. Each component gets collapsed to a single point. 

If this collapsing seems too drastic, one should ponder  maps from rectangles on the surface bounded by arcs of the lamination to rectangles in $\bR^2$, defined as follows: One considers arcs (or isotopy classes of arcs rel endpoints) from a designated corner to a point in the rectangle, and take the intersection numbers with the two foliations to get the two coordinates. This map accomplishes the desired collapsing. For more details, see \cite[Proof of Lemma 6.2]{CB}. 
\end{proof}

Theorem \ref{T:QDMF} is discussed from different points of view in \cite[Section 3]{GM} and \cite[Section 2]{P}. 

\bold{Earthquakes on surfaces.} Consider a simple closed curve $\alpha$ on an oriented hyperbolic surface $X$. The right earthquake for time $t$ about $\alpha$ is the surface $E_{t\alpha}(X)$ obtained by cutting $X$ along the geodesic representative of $\alpha$ and regluing with a twist to the right by hyperbolic distance $t$. The notion of ``right" just depends on the orientation of $X$, and doesn't require any orientation of $\alpha$: Two ants facing each other across the curve $\alpha$ will each see the other move to the right. 

\begin{rem}
$E_{t\alpha}$ determines a flow on Teichm\"uller space. Since we are making a continuous change to the metric, the marking can be transported along the earthquake path. After continuously earthquaking from  $t=0$ to $t=\ell(\alpha)$ (the length of $\alpha$), one arrives back at the same hyperbolic metric, but with a new marking  that differs from the old marking by a Dehn twist. 
\end{rem}

\begin{rem}
In appropriate Fenchel-Nielsen coordinates, $E_{t\alpha}$  is a translation. 
\end{rem}

One can similarly define earthquakes for any simple weighted curve $\alpha$, where the amount of the twist in each curve depends on the measure of a transverse arc. One then defines the earthquake in an arbitrary $\alpha\in \ML$ to be the limit of earthquakes in simple weighted curves $\alpha_n$ that converge to $\alpha$, 
$$E_\alpha(X)=\lim_{n\to\infty} E_{\alpha_n}(X).$$
We will sketch a proof that this is well-defined, i.e. that the limit doesn't depend on the sequence $\alpha_n$ of weighted multi-curves converging to $\alpha$. Our discussion will take in the universal cover. 

\bold{Earthquakes on the hyperbolic plane.} One can define a measured lamination on the hyperbolic plane $\bH$ in the same way as on a surface. However, $\bH$ of course doesn't have any closed curved, so instead we make the following definition: A discrete measured lamination on $\bH$ is a union of disjoint geodesics such that each compact set in $\bH$ intersects only finitely many of the geodesics, and such that each geodesic is endowed with a positive weight  (the transverse measure of a small arc crossing only that geodesic). A key motivating example of a discrete measured lamination on $\bH$ is the preimage of a simple weighted curve on a closed surface.  

Let $\lambda$ be a discrete measured lamination on $\bH$, and let $w_0\in T^1\bH$ be a fixed unit tangent vector. The choice of $w_0$ won't be important and so is usually suppressed from the notation. The earthquake in $\lambda$ is defined to be the unique discontinuous map $E_\lambda: \bH\to \bH$ which is a local isometry off $\lambda$, fixes $w_0$, and ``shears" along each leaf of $\lambda$ by the measure of that leaf. 

The last part of that definition is rather vague, so we'll give a more explicit inductive construction of the earthquake. First, cut out $\lambda$, to get countably many regions of $\bH$, each bounded by geodesics. Define the map to be the identity on the region containing $w_0$. Now, having already defined the earthquake on a region $R$, we explain how to define it for a region $R'$ that borders $R$ along a geodesic $\gamma$ in the support of $\lambda$. Let $m_{\gamma}>0$ be the weight of $\gamma$ in $\lambda$. Let $I_R$ be the isometry of $\bH$ taking $R$ to $E_\lambda(R)$. Let $T$ be the isometry of $\bH$ that translates along $I_R(\gamma)$ by $m_\gamma$. Then we define 
$$E_\lambda(R') = (T\circ I_R) (R').$$
Thus, as in the case of closed surfaces, we arrange for two ants staring at each other from across different sides of $\gamma$ to each see the other move to the right by a distance equal to the weight $m_\gamma$. 


The only reason we  require the earthquake to fix a unit tangent vector is so that it is well defined, avoiding the worry that it might only be well defined up to certain isometries of $\bH$.  

Given any lamination in the hyperbolic plane, we define the earthquake in this lamination to be the limit of earthquakes in discrete laminations which approximate the given lamination. This should again be a discontinuous map $\bH\to \bH$ that is an isometry off the lamination, but as in the surface case we need to show it is well defined and doesn't depend on the choice of discrete approximates.

 \bold{The main estimate.}
We will now sketch a proof that earthquakes are well-defined on $\bH$,  following the more detailed treatment in \cite[Section II]{K}. Some readers might choose to skip this. 

We require two estimates, which refer to the $PSL(2,\bR)$ invariant metric on the unit tangent bundle $T^1 \bH$. We use $E_{tv}$ to refer to the time $t$ earthquake in the geodesic through a unit tangent vector $v$. (A single geodesic with weight 1 is an example of a discrete lamination, so the definition above applies in this case.)

\begin{lem}
For all $D,T>0$ there exists $K=K(D,T)$ such that for all $v, v', w\in T^1 \bH$ that are pairwise distance at most $D$ apart, and all $t\leq T$, we have
$$d(E_{tv}(w), w) \leq Kt$$
and  
$$d(E_{tv}(w), E_{tv'}(w)) \leq Kt d(v, v').$$
\end{lem}

Both estimates are extremely soft, and use only the fact that a differentiable function on a compact set is Lipschitz \cite[Lemma 1.2]{K}. 

Consider two unit tangent vectors $w_0, w$ in $T^1 \bH$ that do not lie on the lamination. We consider two discrete measured laminations $\lambda, \lambda'$ that both approximate the given measured lamination. \emph{We need to show that the  earthquakes corresponding to $\lambda$ and $\lambda'$ that fix $w_0$ do almost the same thing to $w$.} If we can do this, up to small details we will have shown that earthquakes are well defined on $\bH$.

For each discrete approximation ($\lambda$ or $\lambda'$), there is a finite, totally ordered set of geodesics separating the basepoints of $w_0$ and $w$. Only these geodesics and their measures are relevant to understanding the effect of the earthquake on $w$. 

\begin{figure}
\includegraphics[width=.6\linewidth]{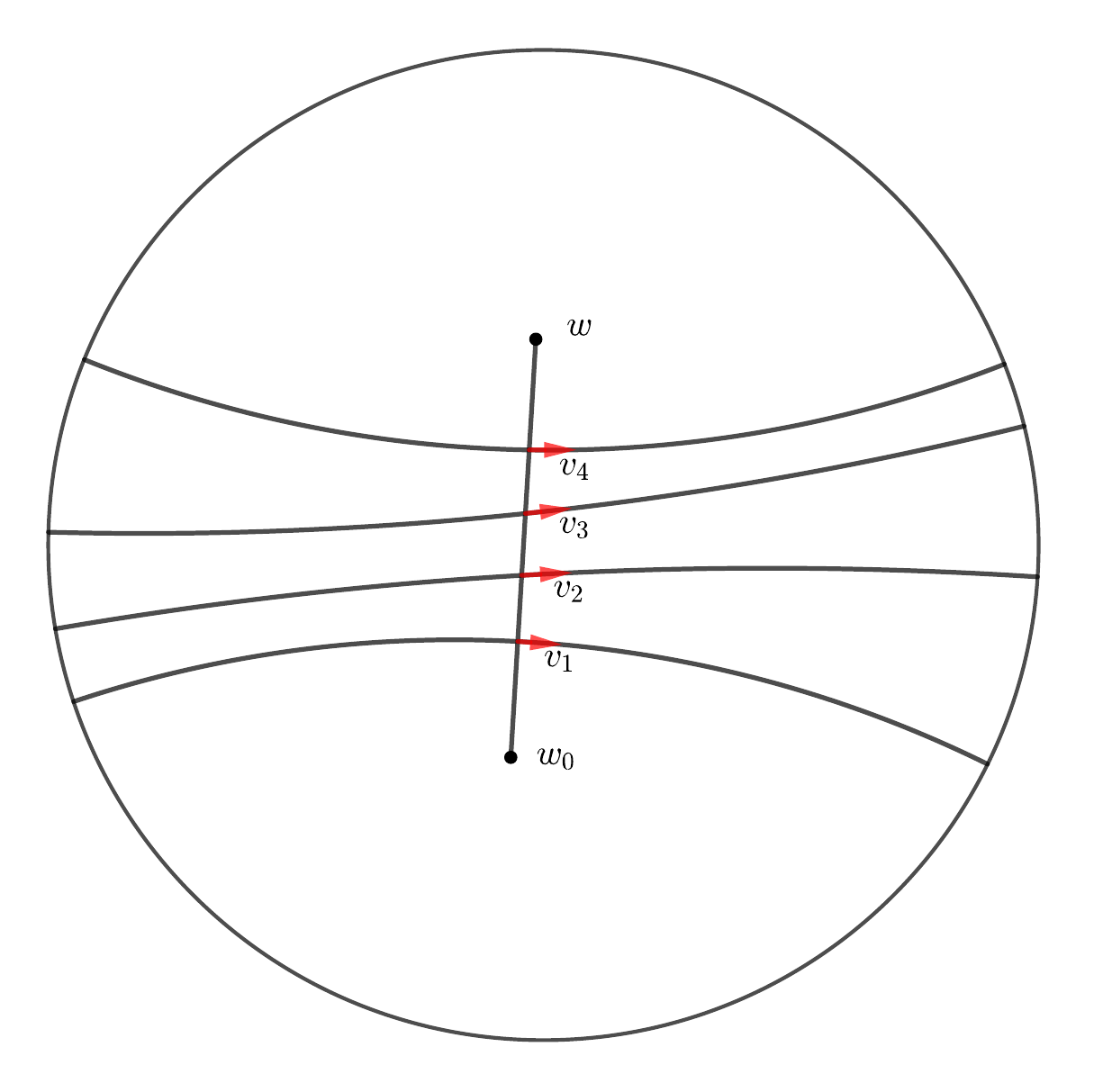}
\caption{The proof that earthquakes are well-defined.}
\label{F:Linear}
\end{figure}

Consider the geodesic arc from the basepoint of $w_0$ to the basepoint of $w$. Consider also the finite sequence $v_i$ (respectively $v_i'$) of unit tangent vectors based at the intersections of $\lambda$ (respectively $\lambda'$) with  this arc that point along $\lambda$ (respectively $\lambda'$). Let $m_i$ (respectively $m_i'$) be the measure (weight) of the geodesic of $\lambda$ (respectively $\lambda'$) along which $v_i$ (respectively $v_i'$) points. 

We now qualitatively outline a quantitative argument in \cite{K}. Let us divide the arc in to small chunks (subintervals). Using the first estimate above, we can reduce to the case that the two discrete measures give exactly the same mass to each chunk. This is because they must give almost the same mass to each chunk, and getting rid of a tiny bit of mass won't change the effect of the earthquake on $w$ much.

\begin{figure}
\includegraphics[width=\linewidth]{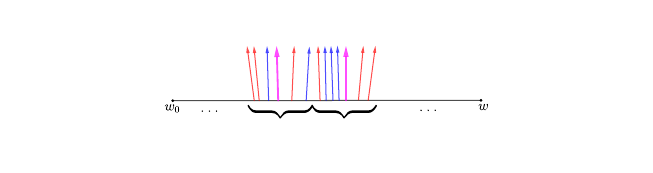}
\caption{The proof that earthquakes are well-defined. The $v_i$ are in red, and the $v_i'$ in blue. On each chunk of the arc from $w_0$ to $w$, they are replaced with a single unit tangent vector $r_k$ shown in purple. }
\label{F:Approx}
\end{figure}

On each chunk, we replace all the $v_i$ in that chunk with a single vector close to all of them, and we earthquake in  that vector with the corresponding amount of mass. The estimates above show that the collective effect of all of these changes is small, so 
this gives that the difference of the two earthquakes applied to $w$ is small. Hence, the earthquakes in $\lambda$ and $\lambda'$ do almost the same thing to $w$. So earthquakes are well defined. 


\bold{Equivariance (back to surfaces).}  If the measured lamination is invariant by a Fuchsian group, the earthquake map will be equivariant by a representation of this Fuchsian group, whose image will be a new Fuchsian group. This new Fuchsian group can be seen as the earthquake of the first Fuchsian group. 

We end by being more explicit about how to get the Fuchsian group representing $E_\lambda(\bH/\Gamma)$ from $\Gamma$. Pick a  $w_0\in T_1 \bH$ not on $\tilde\lambda$. For each $\gamma\in \Gamma$, we consider the earthquake that fixes $w_0$, and pick $\rho(\gamma)$ such that the image of $\gamma(w_0)$ under this earthquake is $\rho(\gamma)w_0$. This is easily seen to be a homomorphism, and we define $E_\lambda(\bH/\Gamma)= \bH/\rho(\Gamma)$. The homomorphism $\rho$ directly defines a marking on $\bH/\rho(\Gamma)$ from a marking on $\bH/\Gamma$, so we get that earthquakes are well-defined on Teichm\"uller space. 

\section{Horocyclic foliations}

A very important construction, which Thurston introduced in \cite{T}, explains how, given a hyperbolic surface $X$ and a certain lamination $\lambda$, we can construct a measured foliation on $X$. Here $\lambda$ should be maximal, i.e., all the complementary regions should be triangles. (Some people call this ``complete" instead of ``maximal"). But $\lambda$ need not support a measure.  

The construction begins by defining the foliation on each complementary triangle using horocycles  based at each ideal vertex.  This gives a foliation of the triangle minus a piece in the center, which can be collapsed to become a three pronged singularity without affecting the foliation along the edges of the triangle.  
\begin{figure}[h!]
\includegraphics[width=.8\linewidth]{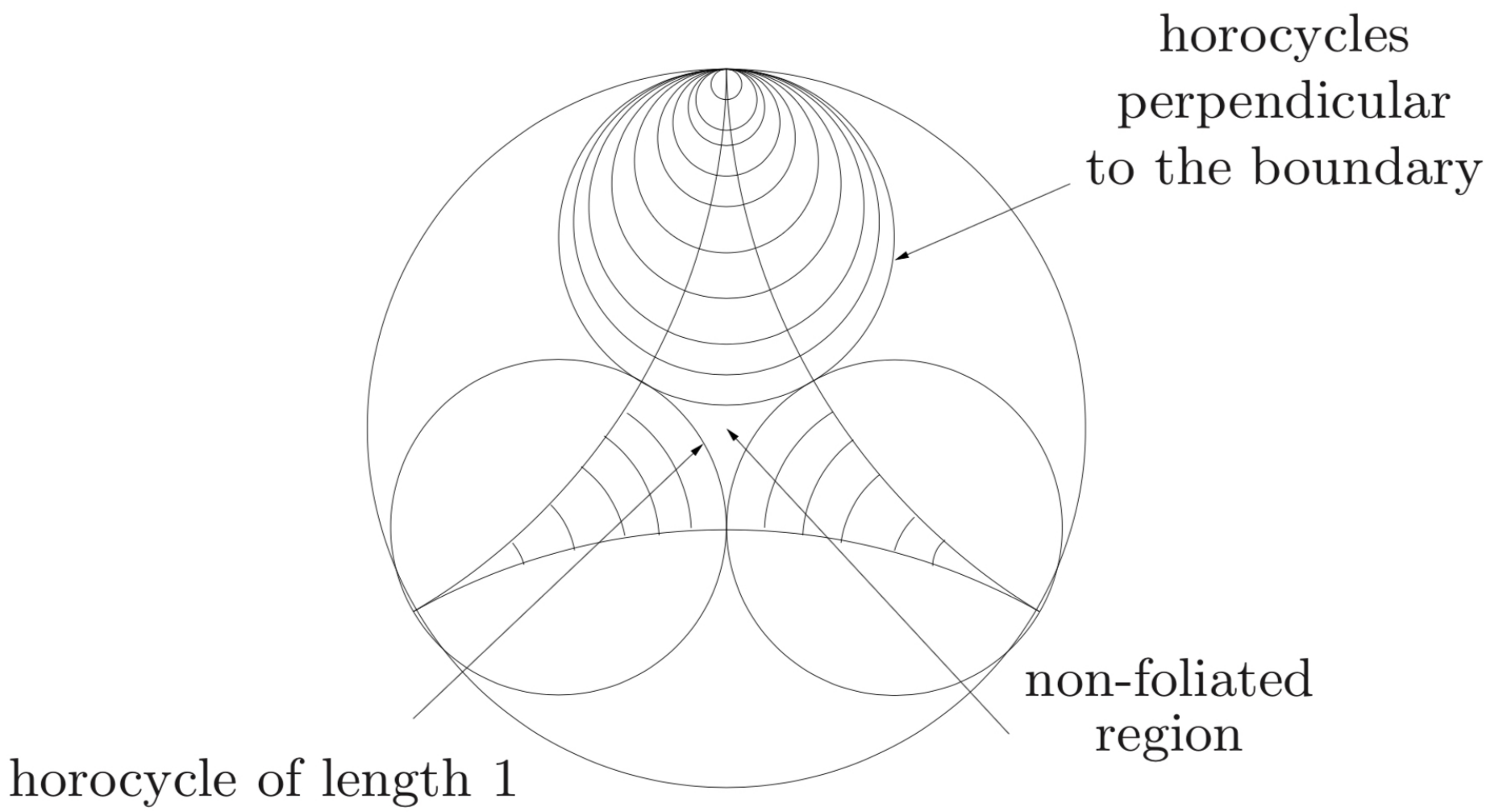}
\caption{A picture from \cite{PapTh} of the horocyclic foliation of a triangle. }
\label{F:foliation}
\end{figure}
The foliation naturally carries a transverse measure in which the measure assigned to the set of leaves passing through a segment of an edge of the triangle is the  length of that segment.
\begin{figure}[h!]
\includegraphics[width=.8\linewidth]{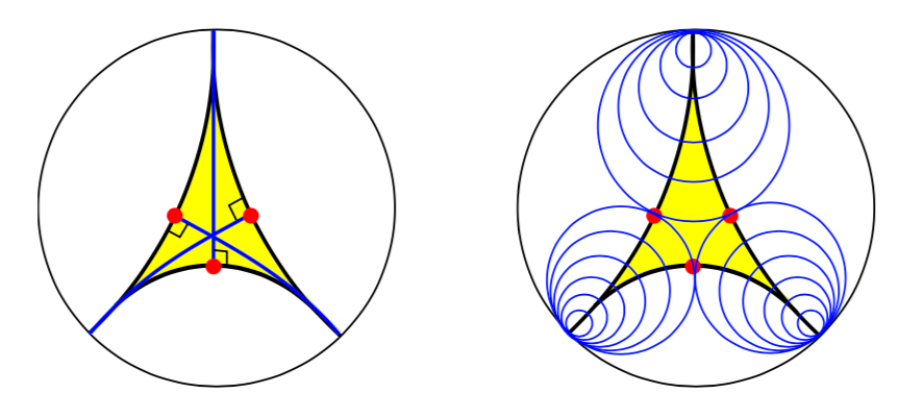}
\caption{A picture from \cite{Mar} of the horocyclic foliation of a triangle. }
\label{F:MarShear}
\end{figure}

In this way we can foliate most of $X$, but the foliation is not yet defined on the vast majority of leaves of $\lambda$ which do not bound complementary regions. However, the partial foliation defined thus far can be checked to be Lipschitz, and hence the associated line field extends continuously to a line field that can be integrated because it is Lipschitz. (One can work with vector fields if desired instead of line fields, for example by working locally.) This gives a map 
$$F_\lambda : \cT_g \to \MF(\lambda),$$
where $\MF(\lambda)$ is the set of measured foliations $\mu$ transverse to $\lambda$, i.e. for which  $(\lambda, \mu)\notin \Delta$. (Because $\lambda,\mu$ are literally transverse, there is an associated quadratic differential whose horizontal and vertical foliations are $\lambda$ and $\mu$, and this implies that $(\lambda, \mu)\notin \Delta$ as discussed above. But one could also just define $\Delta$ to be the set of pairs not associated to a quadratic differential.) 


Our goal in this section is to sketch a proof the following result of Thurston, compare to \cite[Proposition 4.1]{T}. You can choose to accept this theorem as a black box and skip to the next section now. 

\begin{thm}\label{T:shear}
$F_\lambda$ is a homeomorphism
\end{thm}

Often this homeomorphism is followed with a certain map $\MF(\lambda)\hookrightarrow \bR^{6g-6}$ and the result is called shear coordinates for Teichm\"uller space \cite{B}, however we may refer to $F_\lambda$ itself as shear coordinates. 

To prove Theorem \ref{T:shear},  we will explicitely build the inverse of $F_\lambda$. We will build explicitely a hyperbolic surface $X$ whose horocyclic foliation is $\mu$, for any $\mu\in \MF(\lambda)$. 

Imagine we already had such an $X$ with $\mu=F_\lambda(X)$. Then we can lift $\lambda$ to $\tilde{\lambda}\subset\bH$. If $X=\bH/\Gamma$, then $\tilde{\lambda}$ is invariant under $\Gamma$. The idea of the proof is to construct $\tilde{\lambda}$ just from the data of $\mu$. 

To do this it helps to better understand  $\tilde{\lambda}$, assuming $\mu=F_\lambda(X)$. It is this understanding that will allow us to define $\tilde{\lambda}$ in the case when $\mu$ is arbitrary. Let $\tilde{\mu}$ denote the preimage of $\mu$ in $\bH$. 

Consider two triangles $T_1$ and $T_2$ that are complementary regions for $\tilde{\lambda}$. Suppose there is a segment $A$ of $\tilde{\mu}$ that goes from an edge of $T_1$ to an edge of $T_2$, as in Figure \ref{F:T1T2}. 
\begin{figure}[h!]
\includegraphics[width=.5\linewidth]{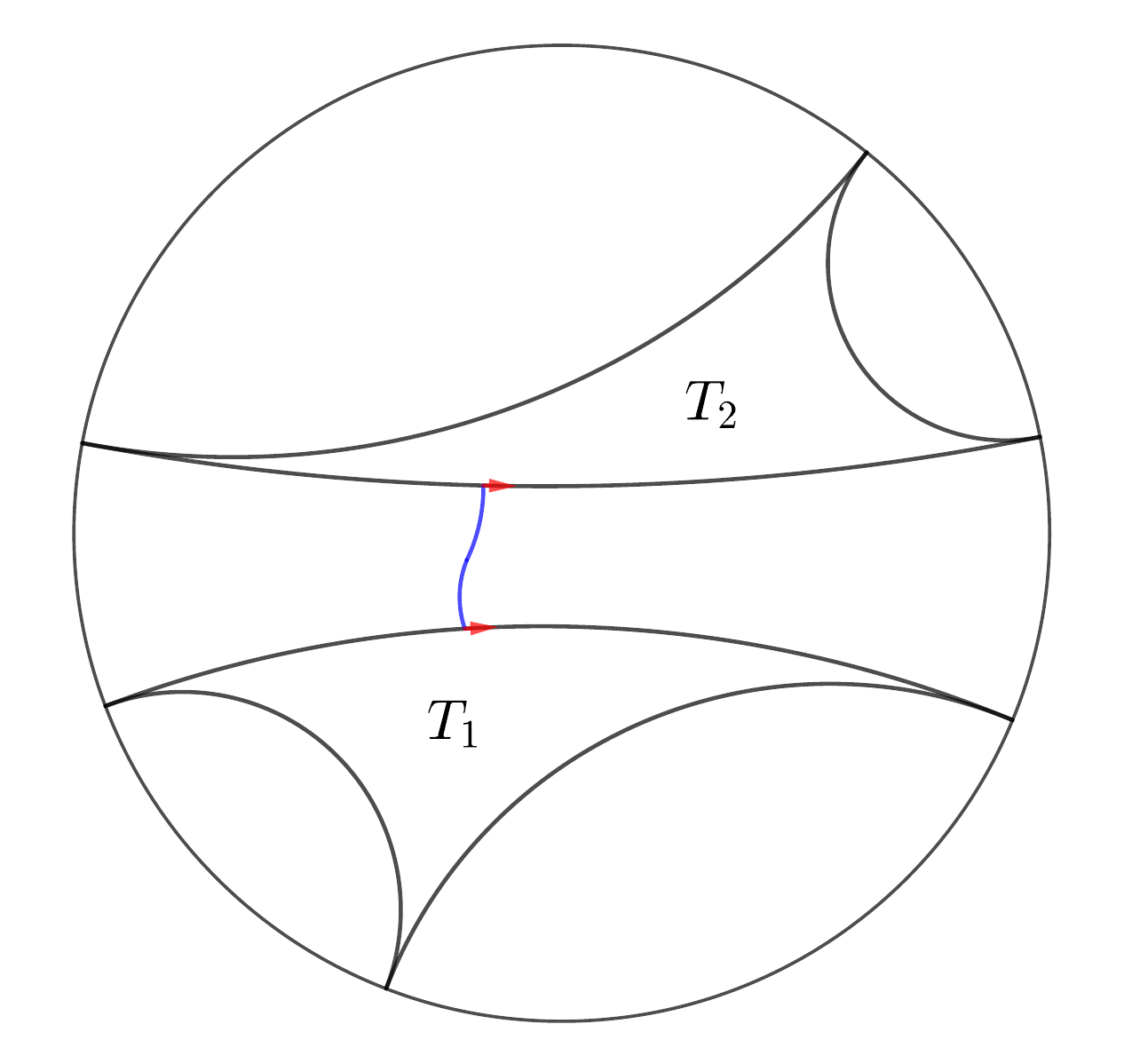}
\caption{}
\label{F:T1T2}
\end{figure}
Consider unit vectors $v_1$ and $v_2$ that are based at the start and end points of $A$ and are tangent to the edges of the triangles. We want to compute the M\"obius transformation $S$, which we view as a two-by-two matrix, that maps $v_1$ to $v_2$. 

This  M\"obius transformation, together with the  ``shear", allows us to recover the position of $T_2$ relative to $T_1$. That is, there is a one parameter family of locations for a triangle $T_2$ with an edge generated by $v_2$, and we call this parameter the shear. The shear can be determined by comparing the distances from the singular leaf in each of the two triangles, as in Figure \ref{F:NewTriangle}. 

\begin{figure}
\centering
  \centering
  \includegraphics[width=.5\linewidth]{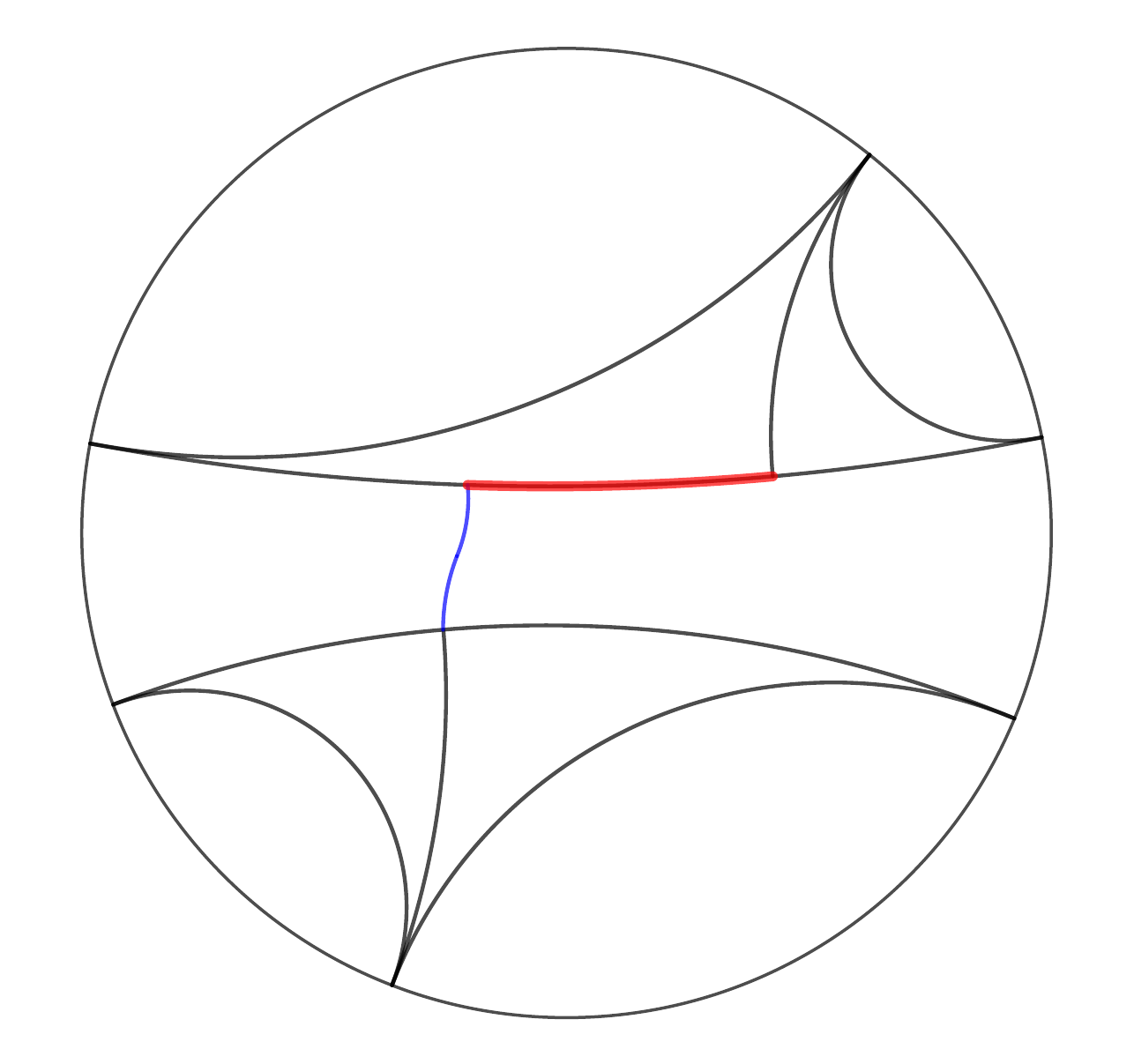}
\caption{The position of $T_2$ relative to $T_1$: Knowing an edge of $T_2$ gives a one parameter family of possibilities for $T_2$. To determine $T_2$, we also need to use the shear, which is the signed length of the red segment. The two half rays in black are orthogeodesics.}
\label{F:NewTriangle}
\end{figure}

Let $I$ be the set of triangles in $\bH$ that are crossed by the segment $A$. Note that $I$ is a countable totally ordered set, but the order is not a well-order. For each $i\in I$, define $v_i^+$ and $v_i^-$ to be the vectors tangent to the edges of the corresponding triangle at the intersection of the edges and $A$, as in Figure \ref{F:viplus}.

\begin{figure}[h!]
\includegraphics[width=.8\linewidth]{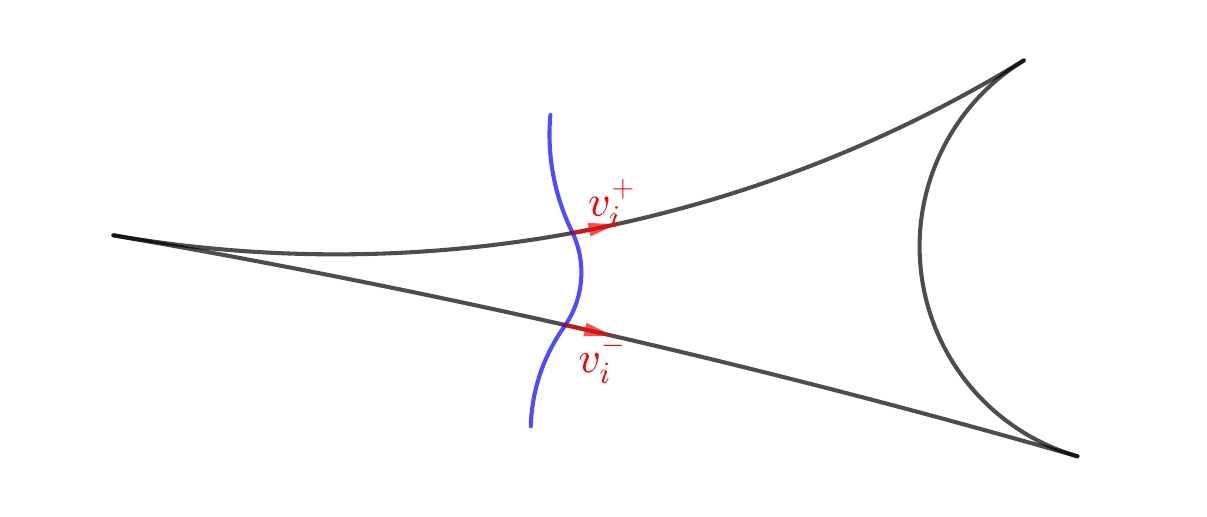}
\caption{The definition of  $v_i^+$ and $v_i^-$.}
\label{F:viplus}
\end{figure}

\begin{figure}[h!]
\includegraphics[width=.8\linewidth]{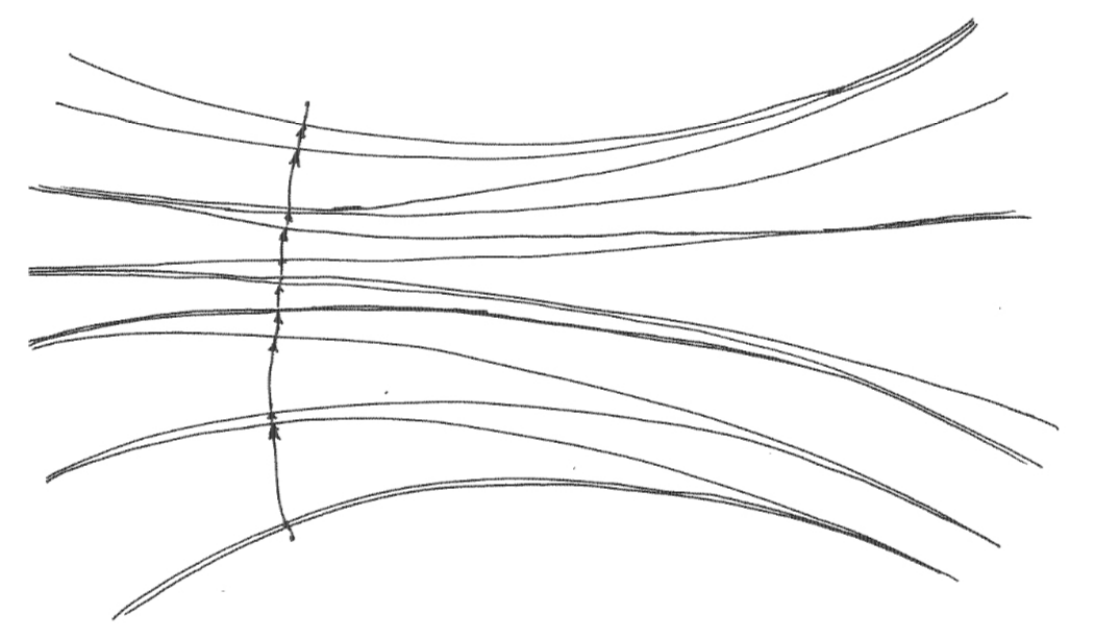}
\caption{Thurston's illustration from \cite{T} of how $A$ crosses $\tilde{\lambda}$. Its intersection with each triangle corresponds to either a stable or unstable horocycle, according to whether the third side of the triangle is to the left or to the right of $A$.}
\label{F:InfiniteProd}
\end{figure}

Let $S_i$ be the M\"obius transformation taking $v_i^-$ to $v_i^+$. We now wish to show that 
$$S = \prod_{i\in I} S_i.$$
That is, the M\"obius transformation moving the vector across infinitely many triangles is the product of the M\"obius transformation moving the vector across each of these triangles. We need a definition to even make sense of what such an infinite product should mean. 

\begin{defn}
Let $I$ be a countable totally ordered set, and let $S_i, i\in I$ be elements of a fixed Banach algebra. Then we say that $\prod_{i\in I} S_i$ is well-defined and equal to $S$ if, for any increasing chain $$I_0\subset I_1 \subset I_2\subset \cdots \subset I$$ of finite sets that exhausts $I$, we have $\lim_{k\to \infty} \prod_{i\in I_k} S_i=S$. 
\end{defn}

The only Banach algebra we will use is the algebra of 2 by 2 matrices, and the only result we will use is the following. 

\begin{lem}
For elements $s_i$ of any Banach algebra indexed by a countable totally ordered set, if $\sum \|s_i\| < \infty$, then $\prod (1+s_i)$ is well-defined. 
\end{lem}
 
\begin{proof}
 Note that for elements $s_1, \ldots, s_n$ of a Banach algebra, and $1\leq m\leq n$, we have 
 \begin{eqnarray*}
 &&\|(1+s_1) \cdots (1+s_n) - (1+s_1) \cdots (1+s_{m-1})(1+s_{m+1}) \cdots (1+s_n)\| 
  \\&&\leq \|s_m\| \prod_{i=1}^n (1+\|s_i\|). 
\end{eqnarray*}
  In the context at hand, the assumption gives that $\prod (1+\|s_i\|)$ is bounded by some constant $C$, so we get that the effect of removing or adding a term $s_n$ is at most $C\|s_n\|$.  
\end{proof}

To apply this lemma, we need to show the two-by-two matrices (M\"obius transformations) $S_i$ that we will use are close enough to the identity. 

\begin{lem}\label{L:summable} 
For the $S_i$ arising as above from $\tilde{\lambda}$ and $A$, if we set $s_i=S_i-1$, then $\sum \|s_i\| < \infty$. (Here $1$ denotes the two-by-two identity matrix.) 
\end{lem}

\begin{proof}
Each $S_i$ can be realized as a time one stable or unstable horocycle flow matrix conjugated by geodesic flow, as in Figure \ref{F:Horo}. 
\begin{figure}[h!]
\includegraphics[width=.8\linewidth]{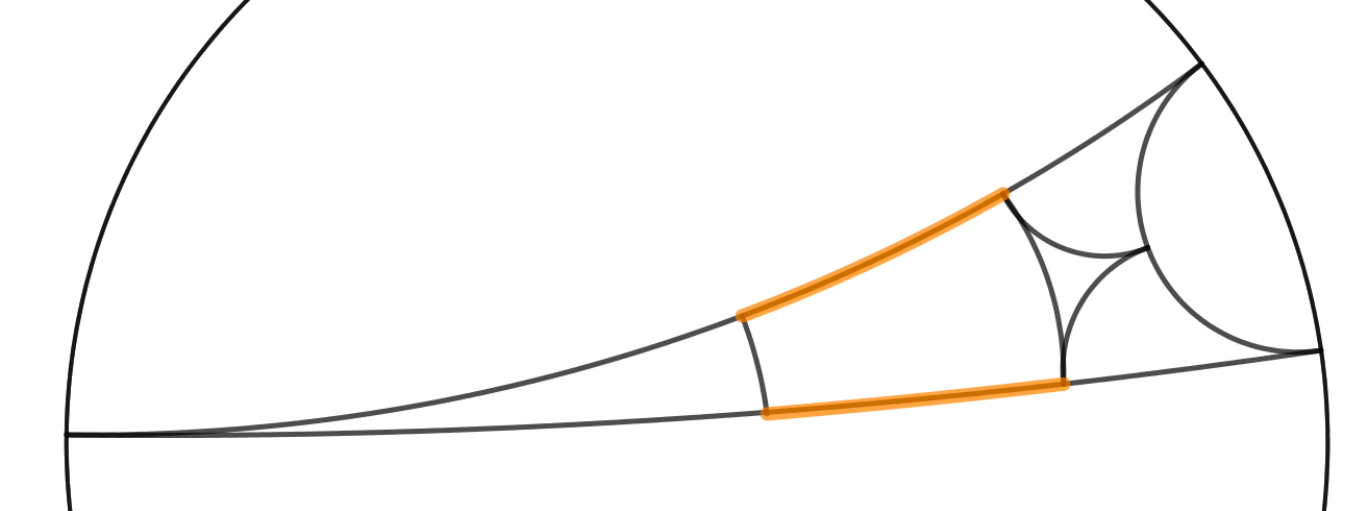}
\caption{$S_i$ can be written as geodesic flow along one orange segment, then horocycle flow for time one, and then geodesic flow backwards along the second orange segment.}
\label{F:Horo}
\end{figure}
The basic computation 
$$ \left(\begin{array}{cc} e^{-t/2}& 0\\ 0 & e^{t/2}\end{array}\right) \left(\begin{array}{cc} 1& 1\\ 0 & 1\end{array}\right) \left(\begin{array}{cc} e^{t/2}& 0\\ 0 & e^{-t/2}\end{array}\right) =\left(\begin{array}{cc} 1& 0\\ 0 & 1\end{array}\right) + \left(\begin{array}{cc} 0& e^{-t}\\ 0 & 0\end{array}\right)$$
shows that the $s_i$ are small whenever the amount of geodesic flow used in the conjugation is large. 

\begin{figure}[h!]
\includegraphics[width=.5\linewidth]{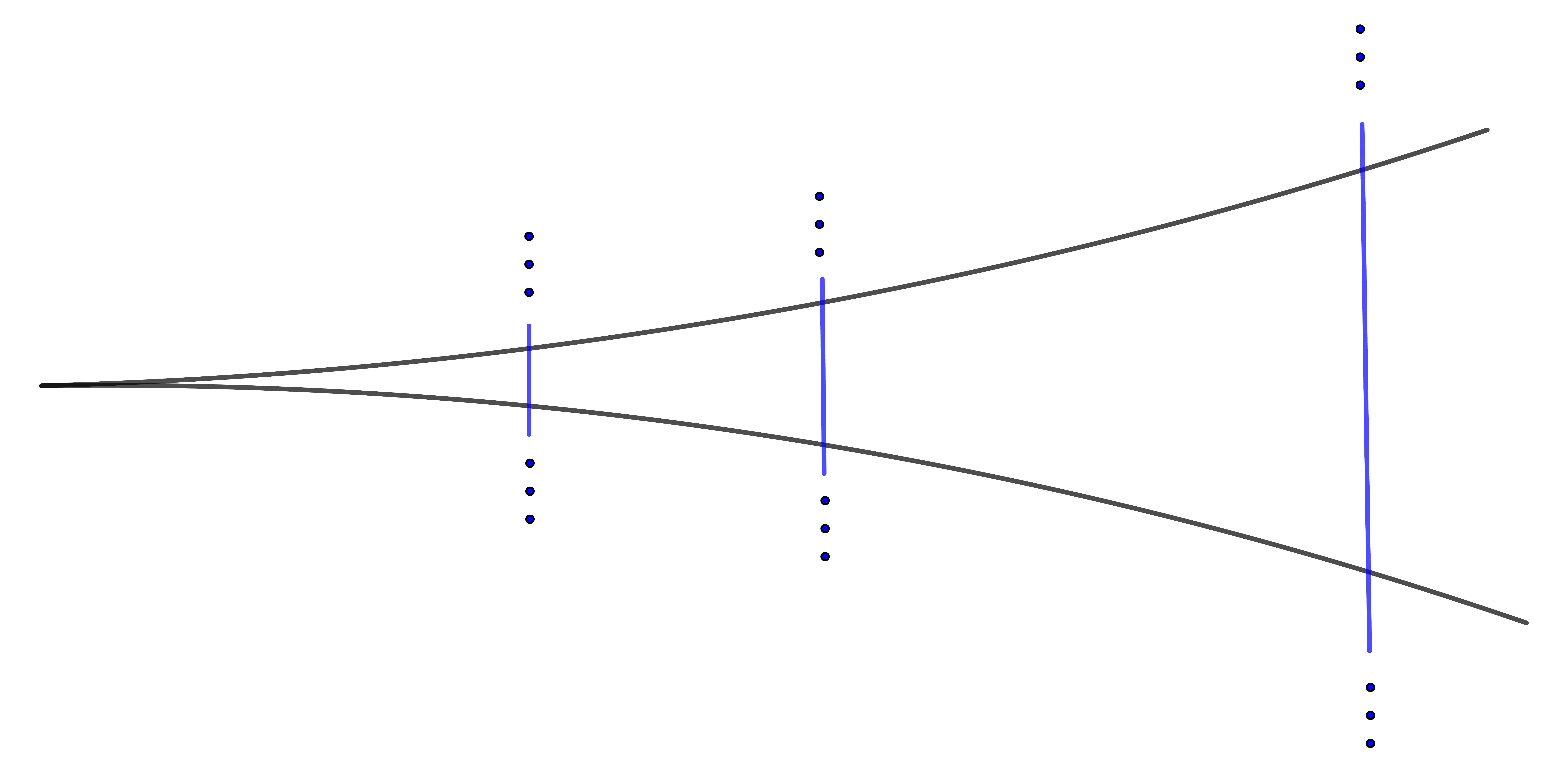}
\caption{A schematic of the intersections of an arc of the foliation with a spike.}
\label{F:Spike}
\end{figure}
We partition all the crossings of our leaf segment $A$ into finitely many subsets according to which ``spike", or corner of a triangle, they cross, see Figure \ref{F:Spike}. Then we show that the sum of the $\|s_i\|$ for each spike is bounded by a geometric progression, because the distance along the spike between neighboring crossings is always bounded below. 
\end{proof}

We now have the desired fact. 

\begin{lem}
$S = \prod_{i\in I} S_i.$
\end{lem}

\begin{proof}
Left to the reader as an exercise. (The hardest parts have been done above for you!)
\end{proof}
%
%

Now, so far we've discussed the relative position of two triangles $T_1$ and $T_2$ which are joined by an arc $A$ of the transverse foliation. Figure \ref{F:Cocycle} shows that not all pairs of triangles are joined by such an arc $A$. The discussion may be clarified then by the following exercise. 

\begin{figure}[h!]
\includegraphics[width=.5\linewidth]{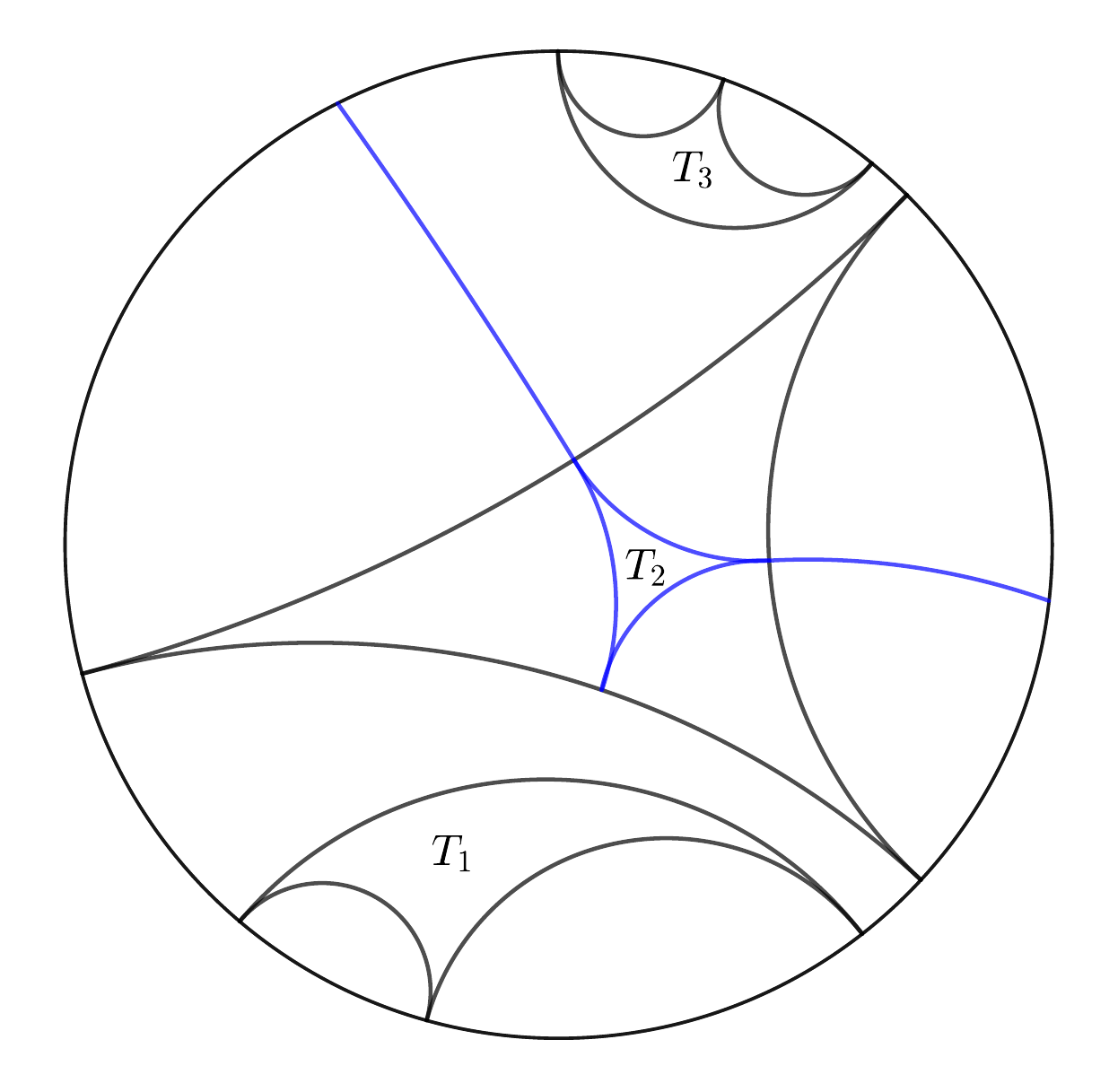}
\caption{$T_3$ is hidden from $T_1$, in that no leaf of the foliation intersects both $T_3$ and $T_1$.}
\label{F:Cocycle}
\end{figure}

\begin{exercise}
For any two triangles $T, T'$ of $\tilde{\lambda}$  there is a sequence of triangles $T=T_0, T_1, \ldots, T_n=T'$ of $\tilde{\lambda}$ such that $T_{i}$ lies in between $T_{i-1}$ and $T_{i+1}$, and there is an arc of $\tilde{\mu}$ from each triangle to the next. 
\end{exercise}

Now we have reached the point where we understand $\tilde{\lambda}$ and $\tilde{\mu}$ quite well, when $\mu=F_\lambda(X)$. In fact, we understand it so well that, from the position of one triangle of $\tilde{\lambda}$, we can exactly determine the positions of all the others using the M\"obius transformations $S$ and the shears. The reader may check their understanding so far by completing the following exercise.   

\begin{exercise}
Convince yourself that the above discussion amounts to a proof that $F_\lambda$ is one-to-one. 
\end{exercise}

\begin{proof}[Proof of Theorem \ref{T:shear}] 
Now we will see that the M\"obius transformations $S$ and the shears can be defined for arbitrary $\mu\in \MF(\lambda)$. In this way we will define $\tilde{\lambda}\subset \bH$, find it is invariant under a group $\Gamma$, and find $F_\lambda(\bH/\Gamma)=\mu$, building the inverse for $F_\lambda$ as desired. 

To start, we use the fact that, given any $\mu\in \MF(\lambda)$, one can isotope $\mu$ to be actually transverse to $\lambda$, and each singularity of $\mu$ is then in a well-defined complementary triangle of $\lambda$ independent of the isotopy. (The singularities of $\mu$ and the complementary triangles of $\lambda$ must be in bijection to each other, because they are both in bijection to the zeros of the associated quadratic differential. Formally speaking, one should write down a more rigorous proof.)   

Even without $X$, there is a topological version of $\tilde{\lambda}$ and $\tilde{\mu}$, defined up to isotopy on the universal cover of the topological surface. They are transverse. 

First, we remark that the shears are obviously defined only in terms of topological data. Indeed, the shear is the transverse measure of the red segment in Figure \ref{F:NewTriangle}. The key point is that whenever we took a hyperbolic length along an arc of a geodesic in $\tilde\lambda$, this was also the transverse measure assigned by $\tilde{\mu}$, because by definition the transverse measure for the horocycle foliation comes from hyperbolic length on the edges of each triangle.  

Next, we recall that each $S_i$ was defined as a conjugate of a time one horocycle flow. The amount of geodesic flow we conjugate by is again a transverse measure assigned by $\tilde{\mu}$, so we can define the $S_i$ from $\mu$ alone. We also check that Lemma \ref{L:summable} applies for arbitrary $\mu$, so we can define the infinite products $S$. 

Now, we can think of placing one triangle $T$ of $\tilde{\lambda}$ on $\bH$ in an arbitrary way. (This arbitrary choice reflects the fact that everything is only defined up to M\"obius transformations.) From this triangle, we can determine where we should put every triangle connected to $T$ by a transverse arc, by using $S(v_1)$ and the shear. Continuing in this way we can determine where we should put every triangle of $\tilde\lambda$. We can obtain the rest of $\tilde\lambda$ as the closure of the set of edges. 

Since the construction arises from objects on the surface, the resulting configuration of triangles in $\bH$ is invariant under a representation of this surface group into $PSL(2,\bR)$. More specifically, for $\gamma$ in $\pi_1$ of the topological surface on which $\mu$ is defined, we may consider a pair of triangles $T_1, T_2=\gamma(T_1)$ in the universal cover. The above discussion computes a M\"obius transformation $\rho(\gamma)$ taking $T_1$ to $T_2$. If $\Gamma$ is the image of $\rho$, then we get that $F_\lambda(\bH/\Gamma)=\mu$ as desired. (Note that $\Gamma$ is discrete because it stabilizes a non-trivial lamination.) This concludes our proof that $F_\lambda$ is a homeomorphism.   
\end{proof}

\begin{rem}
Because of group invariance, we can consider the shear to be defined for any two triangles on $X$ joined by a transverse arc of the foliation. 
\end{rem}

\begin{rem}
If desired one could extend the shear by additivity to all pairs of triangles. For example, in Figure \ref{F:Cocycle}, the shear is defined for $T_1$ and $T_2$, and also for $T_2$ and $T_3$, and we can define the shear between $T_1$ and $T_3$ to the sum of the shears from $T_1$ to $T_2$ and from $T_2$ to $T_3$. This additivity makes it appropriate to refer to the shearing as a cocycle.  
\end{rem}

\section{The Fundamental Lemma on Earthquakes}

We've discussed the shear between two triangles joined by an arc $A$: one follows the singular leaf from one triangle, and looks at where it lands on another triangle, and take  the transverse measure, or equivalently hyperbolic length, of the arc of the boundary geodesic from that landing point to the center point. The fundamental engine of Mirzakhani's isomorphism is how this shear changes when you earthquake in $\lambda$. It is implicit that $\lambda$ is maximal. 

\begin{lem}\label{L:fund}
Denote by $\Shear_X(T_1, T_2)$ the shear for two triangles joined by an arc $A$ of the horocyclic foliation on the hyperbolic surface $X$. Then 
$$\Shear_{E_{t\lambda}(X)}(T_1, T_2) = \Shear_X(T_1, T_2) + t \lambda(A),$$
where $\lambda(A)$ denotes the transverse measure of $A$ and $t$ is sufficiently small. 
\end{lem}

In other words, ``the change in shear is equal to the transverse measure." Mirzakhani cites \cite{B} for this fact, but it can be seen quite easily as follows. 

The restriction that $t$ be small is absolutely not required, but it is sufficient for our purposes, and allows us to avoid thinking about, for example, the situation where $T_2$ and $T_1$ are not joined by an arc of the horocyclic foliation on $E_{t\lambda}(X)$. 

Before reading the proof, the reader may first want to do the following warm up exercise. 

\begin{exercise}
Let $T_1$ and $T_2$ be triangles in $\bH$ that share an edge $\gamma$. How does the shear change after moving one of the triangles by a hyperbolic isometry with axis $\gamma$ and translation distance $t$? 
\end{exercise}

\begin{proof}[Proof of Lemma \ref{L:fund}]
Without loss of generality take $t=1$. 

$T_1$ and $T_2$ are separated by infinitely many leaves of $\tilde\lambda$. As discussed in the definition of earthquakes, we can understand how $T_2$ is moved by the earthquake $E_\lambda$ (assuming $T_1$ is fixed, i.e. relative to $T_1$) by approximating the measured lamination between $T_1$ and $T_2$ by a discrete one. So we do this, picking a discrete lamination consisting of a finite subset of the leaves of $\lambda$ that bound triangles. It doesn't matter to us if this is done in a group equivariant way, since we are just approximating the earthquake in $\bH$. (Indeed the experts may note that it can't be done in a group equivariant way. The quotient would be a discrete lamination, and hence must consist of closed leaves, but $\lambda$ has no closed leaves.) 

If we earthquake along a leaf $\gamma$ of $\tilde\lambda$ between $T_1$ and $T_2$ by an amount $t$, this changes the shear between $T_1$ and $T_2$ by exactly $t$, basically by definition. Indeed, the earthquake applies the hyperbolic isometry that translates along $\gamma$ to the half plane $\bH_\gamma$ on the $T_2$ side of $\gamma$. This moves $\tilde\lambda$ by this isometry on $\bH_\gamma$, and hence it translates the transverse horocyclic foliation on $\bH_\gamma$. Hence, each arc of the transverse horocyclic foliation in $\bH_\gamma$  with an endpoint on $\gamma$ is translated so that the new endpoint is $t$ farther along on $\gamma$. 

Similarly if we earthquake along finitely many leaves of $\lambda$ with measures $t_i$, the shear changes by precisely $\sum t_i$. So, taking a limit, we see that the shear between $T_1$ and $T_2$ changes by an amount equal to the transverse $\lambda$ measure of a transverse arc starting in $T_1$ and ending in $T_2$. 
\end{proof}

\begin{rem}
It may seem strange that the ``Fundamental Lemma", as we have named it, does not apply to arbitrary earthquakes, but only to earthquakes in maximal laminations. However any lamination can be extended, in many ways, to a maximal lamination, and the Fundamental Lemma applies to a measured maximal lamination even if the measure doesn't have full support. Recall that the horocyclic foliation, which we use to define shear, doesn't even depend on or require a measure on $\lambda$. 
\end{rem}


\section{Mirzakhani's isomorphism}

We now turn to the proof of Theorem \ref{T:main}. We begin by specifying full measure sets on which we will build the desired conjugacy $F$.  

Let $QD_0$ denote the locus of quadratic differentials over Teichm\"uller space that don't have any horizontal saddle connections and that have only simple zeros. Recall that a quadratic differential with simple zeros has $4g-4$ zeros. 

\begin{rem}This condition is equivalent to the horizontal lamination being maximal. This can be checked in more than one way. For example, you can note that each simple zero without a horizontal saddle connection gives a complementary triangle, and $4g-4$ times the area of the triangle is the area of the surface, so there is no room for any other complementary regions. 
\end{rem}

Let $\ML_{0}$  denote the locus of measured foliations that are maximal. We will build a mapping class group equivariant measurable isomorphism $F$ from $\ML_0\times \cT_g$ to $QD_0$ that conjugates earthquake flow to unipotent flow. The map sends $(\lambda,X)$ to the quadratic differential with foliations $(\lambda, F_\lambda(X))$, 
$$F(\lambda,X)=q(\lambda, F_\lambda(X)).$$ 
Here $F_\lambda(X)$ continues to denote the horocyclic foliation of $X$. This map is is only measurable, but its restriction to each slice $\{\lambda\}\times \cT_g$ is a homeomorphism onto the set of quadratic differentials with horizontal lamination $\lambda$ by Theorem \ref{T:shear}. It follows that $F$ is a bijection from $\ML_0\times \cT_g$ to $QD_0$. 

It remains only to show that the image of the earthquake flow path $(\lambda, E_{t\lambda}(X))$ is a unipotent flow path.  We begin by discussing Teichm\"uller unipotent flow, which is of course characterized by how it changes period coordinates. But first we present a lemma that will allow us to restrict from arbitrary periods to special saddle connections. 

\begin{lem}\label{L:rectangles}
Every isotopy  class of path joining singularities of a  quadratic differential can be realized by a sequence of paths that start at one singularity, travel in the horizontal direction, then travel in the vertical direction and end at a singularity.
\end{lem}

\begin{proof}
It suffices to prove this for saddle connections. This can be done by growing rectangles: Look at the rectangle from one endpoint to a point on the saddle connection nearby, and grow this rectangle until it hits a singularity. Continue in this way as in Figure \ref{F:Rectangles}.
\begin{figure}[h!]
\includegraphics[width=.5\linewidth]{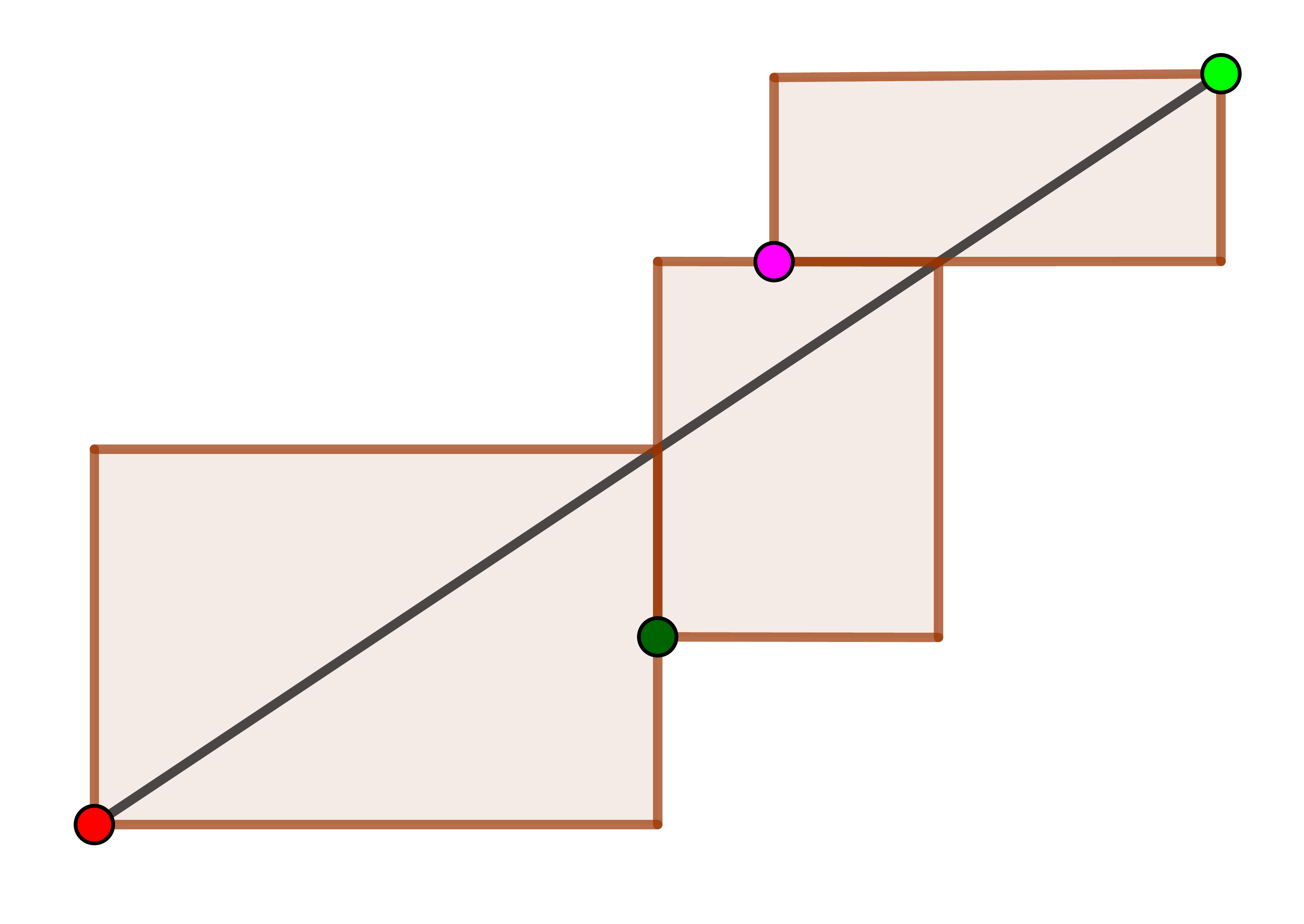}
\caption{The proof of Lemma \ref{L:rectangles}.}
\label{F:Rectangles}
\end{figure}
\end{proof}

\begin{cor}\label{C:unipotent}
Suppose $q_t$ is a path of quadratic differentials. Suppose that for every $t_0$ and every path $\gamma$ on $q_{t_0}$ as in the lemma, the period $x_t+iy_t$ of $\gamma$ satisfies 
$$\left. \frac{d}{dt}\right|_{t=t_0}x_t=y_{t_0}, \quad\text{\and}\quad \left. \frac{d}{dt}\right|_{t=t_0}y_t = 0.$$
Then $q_t$ is an orbit of Teichm\"uller unipotent flow. 
\end{cor}

\begin{proof}
It suffices to recall that 
\begin{eqnarray}\label{E:fund}
\left(\begin{array}{cc} 1& t\\ 0 & 1\end{array}\right) \left(\begin{array}{c} x\\ y \end{array}\right) = \left(\begin{array}{c} x+ty\\ y \end{array}\right)
\end{eqnarray} 
and that a function with constant derivative is linear.
\end{proof}

Observe that  the $y$ component of the period of $\gamma$ is given by the transverse measure of $\gamma$ for the horizontal measured foliation, see Figure \ref{F:y}.
\begin{figure}[h!]
\includegraphics[width=.5\linewidth]{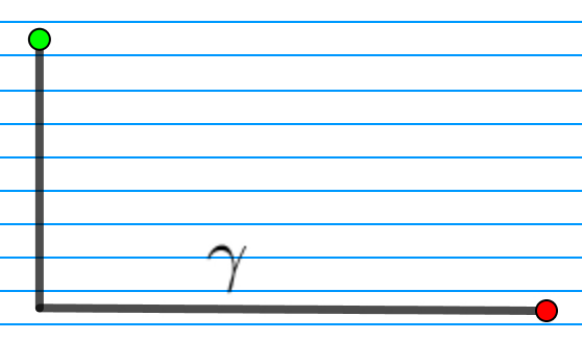}
\caption{The $y$ component is given by the horizontal foliation.}
\label{F:y}
\end{figure}
The intuition of the proof of Theorem \ref{T:main} is to think of each singularity of $q$ as corresponding to a complementary triangle for a lamination, and to think of the $x$ component of a period of such a $\gamma$ as the shear between the two corresponding triangles. We offer the following chart to summarize this intuition, before beginning the formal proof. 

\vspace{.5cm}
\begin{center}
\begin{tabular}{ c c c }\vspace{.1cm}
  earthquake flow & $\longleftrightarrow$ & horocycle flow \\ \vspace{.1cm}
  $\lambda$ & $\longleftrightarrow$ & horizontal foliation \\ \vspace{.1cm}
  $F_\lambda(X)$ & $\longleftrightarrow$ & vertical foliation \\ \vspace{.1cm}
 $(\lambda, F_\lambda(X))$ & $\longleftrightarrow$ & quadratic differential \\ \vspace{.1cm}
  triangle  & $\longleftrightarrow$ & singularity \\ \vspace{.1cm}
  shear  & $\longleftrightarrow$ & $x$-component of a period \\ \vspace{.1cm}
  Fundamental Lemma  &$\longleftrightarrow$& equation (\ref{E:fund})\\
\end{tabular}
\end{center}
\vspace{.5cm}

\begin{proof}[Proof of Theorem \ref{T:main}]
We wish to show that $$q(\lambda, F_\lambda(E_{t\lambda}(X)))$$ is a horocycle flow path using Corollary \ref{C:unipotent}. We will consider a moment in time $t_0$, which without loss of generality  is $t_0=0$, and show that for each path $\gamma$ as above, the derivative of the period of $\gamma$ satisfies Corollary \ref{C:unipotent}. 
\begin{figure}[h!]
\includegraphics[width=.7\linewidth]{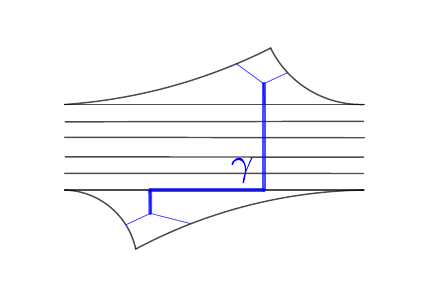}
\caption{This picture isn't geometrically accurate, but it gives an idea of how to think of $\gamma$ as it lies on $\bH=\tilde{X}$. The  horizontal lines  are leaves of $\tilde\lambda$, and the vertical lines are leaves of the horocyclic foliation.}
\label{F:gamma2}
\end{figure}

We've already done most of the work to see this. Indeed, the path $\gamma$ corresponds to a path in $X$ or $\tilde{X}=\bH$ joining two triangles. If the period of $\gamma$ is $(x_t, y_t)$, then we see that $x_t$ is the transverse measure of $\gamma$ given by the vertical foliation, and similarly for $y_t$. So $y_t=\lambda(\gamma)$ is constant. And the derivative of $x_t$ is equal to  $y_t=\lambda(\gamma)$ by the Fundamental Lemma. 

Hence Corollary \ref{C:unipotent} gives that  $$q(\lambda, F_\lambda(E_{t\lambda}(X)))= F(\lambda, E_{t\lambda}(X))$$ is an earthquake path as desired. We already known that $F$ is a homeomorphism from $\ML_0\times \cT_g$ to $QD_0$, so this concludes the proof. 
\end{proof}

\begin{rem}\label{R:differences}
In \cite{M}, Mirzakhani factors the map $F_\lambda$ through a subset $H_+(\lambda)$ of the space $H(\lambda,\bR)$ of transverse cocycles on $\lambda$. There is a map 
$$I_\lambda:\ML(\lambda)\to H_+(\lambda)$$
 that records the shearing data as a transverse cocycle. Mirzakhani considers a map 
$$G_\lambda:\cT_g\to H_+(\lambda)$$
satisfying $F_\lambda= I_\lambda^{-1} \circ G_\lambda$. The map $G_\lambda$ is a symplectomorphism by \cite{BS}, and using the fact the earthquake flow is Hamiltonian  Mirzakhani concludes that $G_\lambda$ conjugates $E_{t\lambda}$ with a natural linear flow on $H_+(\lambda)$.  Mirzakhani also notes that if $q(\lambda, \mu_t)$ is a horocycle flow path, then $I_\lambda(\mu_t)$ is an orbit of the same linear flow on $H_+(\lambda)$. 
From this she concludes that $F$ conjugates earthquake flow to horocycle flow. The proof we have presented thus differs from \cite{M} in that we do not factor through the intermediary $H_+(\lambda)$ and we use the Fundamental Lemma, rather than the Hamiltonian nature of the flows and \cite{BS}, to establish the conjugacy. 
\end{rem}

\begin{rem}
I conjecture that the semi-conjugacy is continuous on $\ML_0\times \cT_g$. For example, if you take a sequence of maximal laminations $\lambda_n$ that converge to some $\lambda$ that is also maximal, then for each fixed $X$ the horocyclic foliation on $X$ for $\lambda_n$ should converge to that of $\lambda$. 

But the semi-conjugacy  cannot be extended continuously to even to the locus where $\lambda$ has a single quadrilateral and the rest of the complementary regions are triangles. The reason is that such $\lambda$ are limits of maximal laminations in two different ways, essentially corresponding to the two different ways to turn the quadrilateral into two triangles. These two choices give two different horocyclic foliations. 
\begin{figure}[h!]
\includegraphics[width=.5\linewidth]{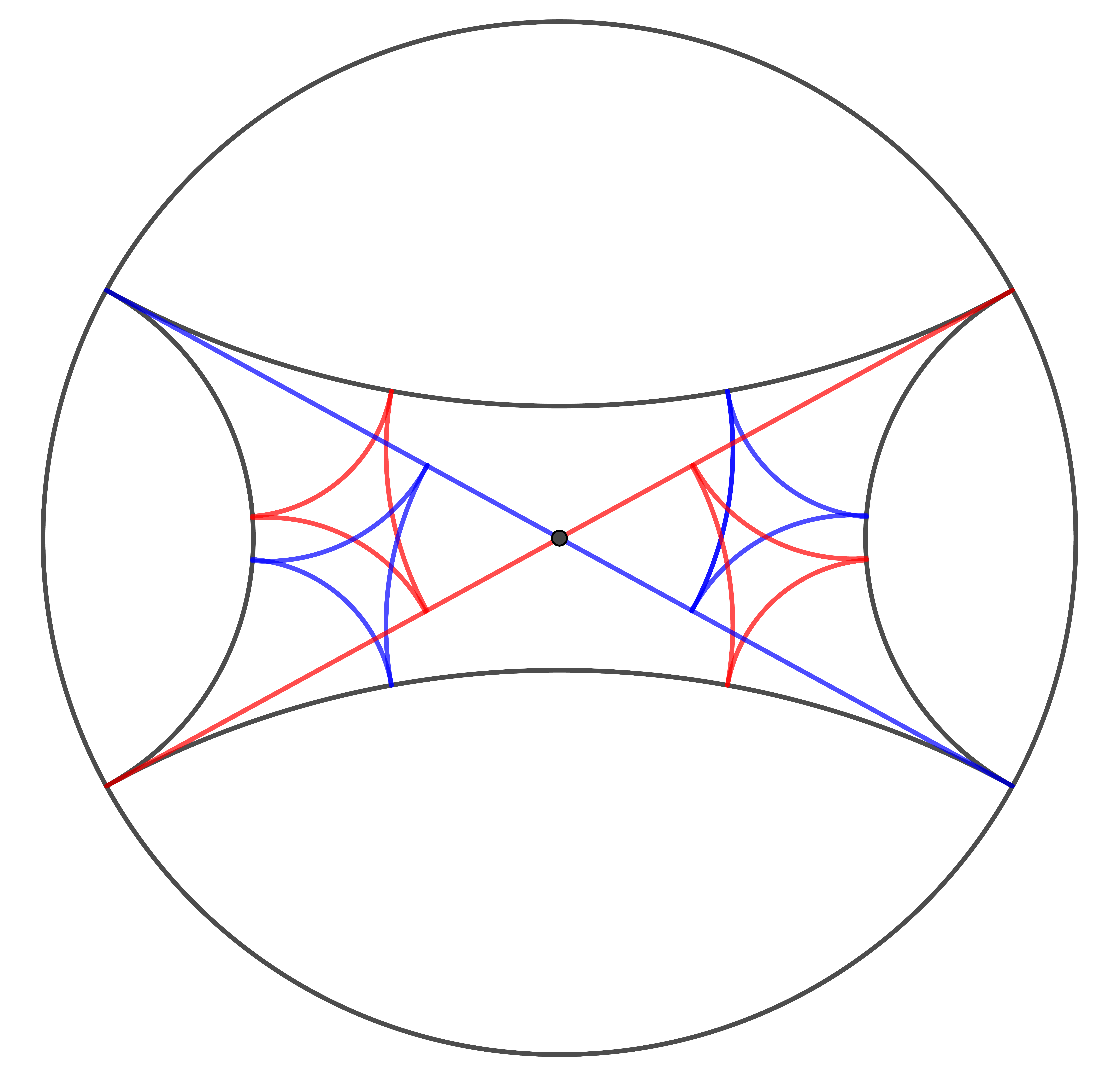}
\caption{The proof that $F$ cannot be extended to a continuous map.}
\label{F:NotContinuous}
\end{figure}
Note that this situation arises if you approximate a quadratic differential with a saddle with nearby quadratic differentials where that saddle can either slope slightly up or slightly down. (This comment is necessary because you can't just add one geodesic to the quadrilateral and get a lamination with a transverse measure of full support.) 

In general, for each $\lambda$, there is possibly a finite or infinite number of ways to fill in $\lambda$ to a maximal lamination (without a measure of full support), and each of these different maximal extensions gives a different horocyclic foliation that will serve as the vertical measured foliation for a quadratic differential. Perhaps one can think that Mirzakhani's map as being multivalued off of $\ML_0$ and the multiple values correspond to all these choices of maximal extension. Similarly if one wanted to compute $F_\lambda$ when $\lambda$ isn't maximal, one could do it by extending $\lambda$ to be maximal in a number of ways, so one can also think of the inverse of $F$ as being multi-valued. 

Alternatively, one could consider earthquake flow on $\ML_{ext}\times \cT_g$, where $\ML_{ext}$ consists of all pairs of a measured lamination plus an extension of its support to a maximal lamination, and we use the topology that requires convergence of both the measured lamination and the maximal unmeasured extension. This flow should map continuously onto both earthquake flow and Teichm\"uller unipotent flow.  
\end{rem}

\begin{rem}\label{R:continuous}
I don't know how to show that there couldn't be some (totally different) continuous conjugacy between earthquake flow and Teichm\"uller unipotent flow. This seems like an interesting open problem. 
\end{rem}

\begin{rem}\label{R:strat}
Consider a partition $\kappa$ of $4g-4$.
Consider the subset $\cI(\kappa)\subset \ML\times \cT_g$ where the complementary regions of the maximal measured lamination are all symmetric ideal hyperbolic polygons, with the number of polygons with a given number of edges given by $\kappa$. (You view $\kappa$ as a partition of the area divided by $\pi$.) Minus the symmetry assumption, this would be just a condition on the measured lamination, and not the point of Teichm\"uller space. The symmetry condition says that there is a hyperbolic isometry that cyclically permutes the ideal vertices. 

One can't even extend $F$ to this locus. However, I conjecture that there is a different $F$ that is a conjugacy from $\cI(\kappa)$ to the stratum $\cQ(\kappa)$ of quadratic differentials. To be more precise, the image of $F$ would be the locus with no horizontal saddle connections in that stratum. This map would use the horocyclic foliation that is defined for symmetric ideal polygons. 
\end{rem}

%
\begin{rem}\label{R:length}
There is a notion of hyperbolic length of geodesic laminations. The hyperbolic length of $\lambda$ on $X\in \cT_g$, is $i(\lambda, F_\lambda(X)))$. If you'd like you can take this as a definition. It makes sense because the transverse measure for $F_\lambda(X)$ corresponds to arc-length along $\lambda$.
 
It follows that the semi-conjugacy is such that if $\lambda$ has hyperbolic length $\ell$ on $X$, then the resulting quadratic differential has area $\ell$. A fancy way of putting this is to say $\lambda$ has extremal length $\ell$ on the image Riemann surface. (Don't worry if you don't know what that means.)
\end{rem}

\begin{rem}
Mirzakhani's map $F$ simultaneously conjugates Thurston's stretch map flow \cite{T} to the action of 
$$\left(\begin{array}{cc} 1& 0\\ 0 &e^s\end{array}\right)$$
on $QD$. This is because Thurston's stretch map flow is simply scalar multiplication in shear coordinates. 
\end{rem}

\begin{rem}
I conjecture that $F$ maps co-bounded sets to co-bounded sets. Co-bounded means contained in a compact set after you quotient by the action of the mapping class group. This is related to the fact that the set of maximal unmeasured laminations on a given (unmarked) hyperbolic surface should be compact. So given any non-maximal lamination, there should be a compact set of ways to extend it to a maximal lamination, and these should correspond to the possible limiting values of $F$. Similarly for $F^{-1}$. 
\end{rem}

\section{Invariant measures}

There is a natural measure called Thurston measure $\mu_{Th}$ on $\ML$. It is basically the same thing as Masur-Veech measure. Most laminations are not orientable, but  can be made so by passing to a double cover, after which they give a cohomology class. For nearby laminations, you can pass to a common (branched) double cover, so they give cohomology classes in the same vector space. Thurston measure is Lebesgue measure in this vector space. (Actually the vector space is the $-1$ eigenspace of the double cover.) 

As discussed, $QD$ is equal to $\ML\times \ML\setminus \Delta$. It isn't hard to show that the Masur-Veech measure (not just on the unit area locus) is equal to the restriction of $\mu_{Th}\times \mu_{Th}$ to the complement of $\Delta$. Indeed, Masur-Veech measure also arises from taking cohomology classes on the double cover where the foliation becomes orientable. 

A basic fact that we will discuss in the next section is that earthquakes are Hamiltonian flows. A corollary is the following. 

\begin{thm}
The action of $E_{t\lambda}$ on $\cT_g$ leaves invariant the Weil-Petersson measure $\mu_{WP}$
\end{thm}

Recall that the  Weil-Petersson is nothing other than the standard Lebesgue measure in Fenchel-Nielsen coordinates. 

\begin{cor}
For any measure $\nu$ on $\ML$, the earthquake flow leaves the measure $\nu\times \mu_{WP}$ on $\ML\times \cT_g$ invariant.  In particular, $\mu_{Th}\times \mu_{WP}$ is both invariant under earthquake flow and the action of the mapping class group. 
\end{cor}

Recall from Remark \ref{R:length} that one can take the hyperbolic length of a lamination. It isn't hard to show that this length $\ell_X(\lambda)$ is invariant under earthquake flow in $\lambda$.  For example, when you earthquake in a simple closed curve, the hyperbolic length of that curve doesn't change.  So earthquake flow preserves each level set $\ML_\ell$ for the hyperbolic length of $\lambda$. 

If one wishes an invariant measure on the set $\ML_1\times \cT_g$ where the measured lamination has length 1, one does the same thing as for Masur-Veech measure. Namely, over a point $X\in \cT_g$, the measure used on $\ML_1$ gives a subset of $\ML_1$ the Thurston measure of its cone in $\ML$. (The cone on a set consists of anything in the set times any number in $[0,1]$.) This gives a mapping class group and earthquake flow invariant measure on $\ML_1\times \cT_g$. 

Similarly one gets invariant measures on any other level set $\ML_\ell\times \cT_g$ for the length function. The measure of $\ML_\ell\times \cT_g$ is equal to $\ell^{6g-6}$ times the measure of $\ML_1\times \cT_g$ (the measure is finite after quotienting by the mapping class group).

The measure on $\ML_\ell\times \cT_g$ must map to a measure  on $QD_\ell$, the set of area $\ell$ quadratic differentials. The image measure is Lebesgue class  and invariant under horocycle flow, so using ergodicity of horocycle flow it must be a multiple of Masur-Veech. (The first point should be true since both maps $\ML\times \cT_g\to \ML\times \MF$ and $QD\to \ML\times \MF$ are pretty nice and understandable maps. For example, if you change $\lambda$ a bit then $F_\lambda(X)$ only changes a bit. For a formal proof, one likely has to look at Bonahon's papers.) 

The isomorphism is also a conjugacy for rescaling the $\lambda$ and rescaling the horizontal foliation of the quadratic differential. The measure of $QD_\ell$ is also equal to $\ell^{6g-6}$ times the measure of $QD_1$, so we get the multiple is independent of $\ell$. Hence $\mu_{Th}\times \mu_{WP}$ maps to $c\mu_{Th}\times\mu_{Th}$ for some $c>0$, and so $F_\lambda^*(\mu_{Th})=c\mu_{WP}$. In fact,  Bonahon-Sozen gave a more explicit proof of this, before Mirzakhani's isomorphism, that computes that $c=1$ and handles the symplectic forms rather than just their associated volume forms \cite{BS}.

\begin{thm}
$F_\lambda^*(\mu_{Th})=\mu_{WP}$.
\end{thm}

 Their proof was discovered using the case when $\lambda$ contains a pair of pants (which doesn't fit into our setting, since such $\lambda$ don't have fully supported transverse measures) and the magic formula of Wolpert. 

\begin{cor}
The Masur-Veech volume of the principal stratum of quadratic differentials is 
$$\int_{X\in \cM_g} \mu_{Th}(B_X(1)) d\mu_{WP},$$
where $B_X(1)$ is the unit ball in $\ML$ of lamination of length at most 1 on $X$. 
\end{cor}

This corollary is \cite[Theorem 1.4]{M}. 

\begin{proof}[Proof sketch]
Using Fubini, the integral is the $\mu_{WP} \times \mu_{Th}$ measure of a fundamental domain for the mapping class group on 
$$\{(X, \lambda) \in \cT_g \times \ML : \ell_{X}(\lambda)\leq 1\}.$$
Using that $F_\lambda^*(\mu_{Th})=\mu_{WP}$, this is equal to the $\mu_{Th}\times \mu_{Th}$ measure of a fundamental domain for the mapping class group on 
$$\{(h,v) \in \ML\times \ML\setminus \Delta : i(h,v)\leq 1\},$$
which is the Masur-Veech volume of the principal stratum.
\end{proof}

\section{Laminations containing a pants decomposition}

We now consider the case of maximal lamination $\lambda$ that contains a pants decomposition $\cP$, i.e. a maximal set of disjoint curves. Such $\lambda$ are seemingly irrelevant for the discussion above, because they are not in the locus where Mirzakhani's semi-conjugacy is defined. Indeed, because such $\lambda$ can't have a fully supported transverse measure, they can't arise as the  horizontal lamination of a quadratic differential. But the map $F_\lambda$ is defined for any $\lambda$ maximal, and considering this case will provide insight. 

We can glue together two topological ideal triangles, i.e. triangles minus their vertices, to get a sphere minus three points, as in Figure \ref{F:TopPants}. 

\begin{figure}[h!]
\includegraphics[width=.25\linewidth]{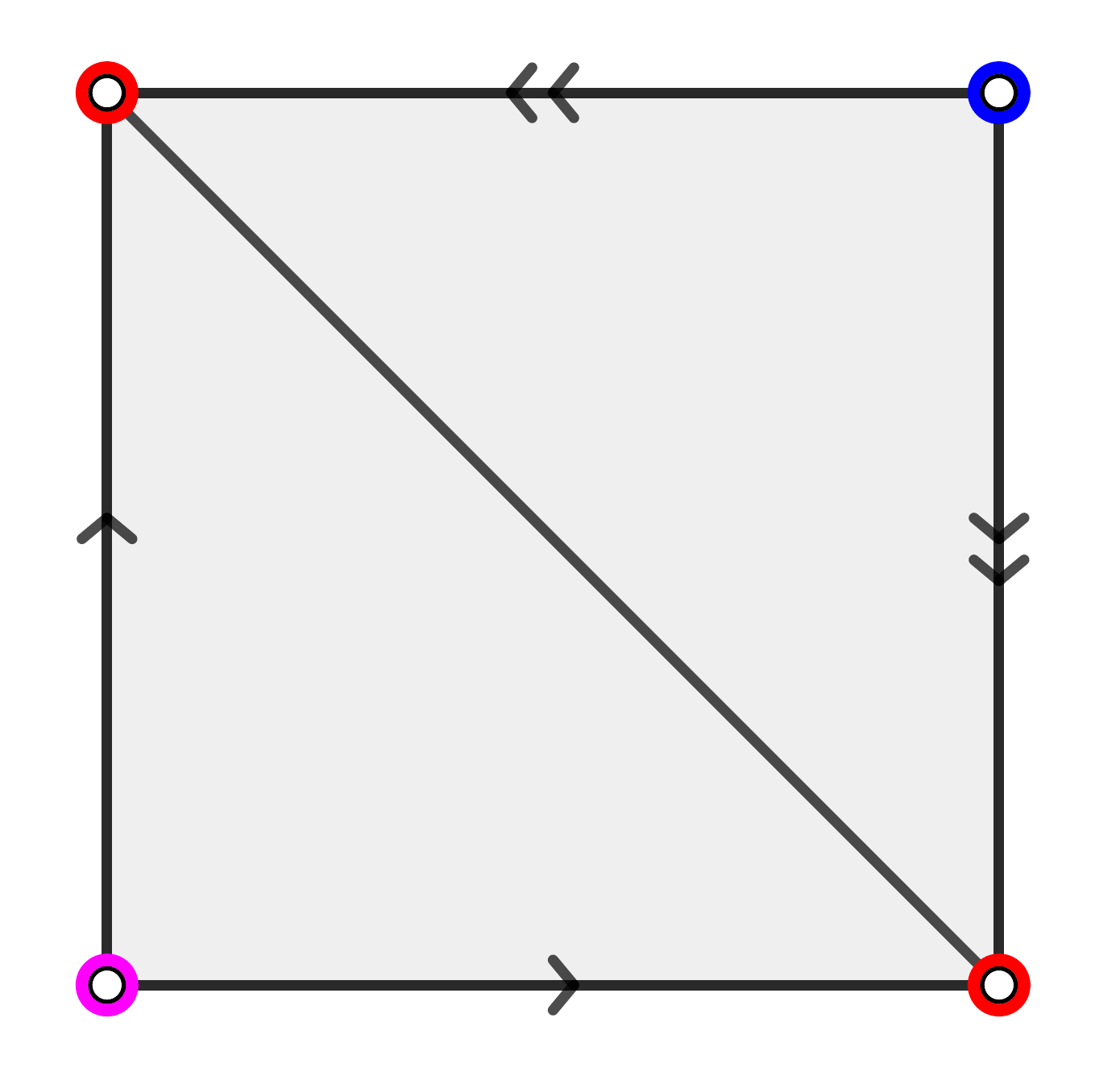}
\caption{A sphere minus three points can be obtained by gluing two triangles minus their vertices. }
\label{F:TopPants}
\end{figure}

Let us consider gluing together two ideal hyperbolic triangles along isometries of their edges. We'll glue in the same pattern, so we know that the result will topologically be a sphere minus three points, which topologically is the same thing as a sphere minus three discs. The result will have a hyperbolic metric, but this metric might be incomplete. There are three parameters, the three shears, that we'll denote $s_1, s_2, s_3$. Each shear is the distance between ``center points" of edges that are glued together, in the usual way.  

\begin{figure}[h!]
\includegraphics[width=.35\linewidth]{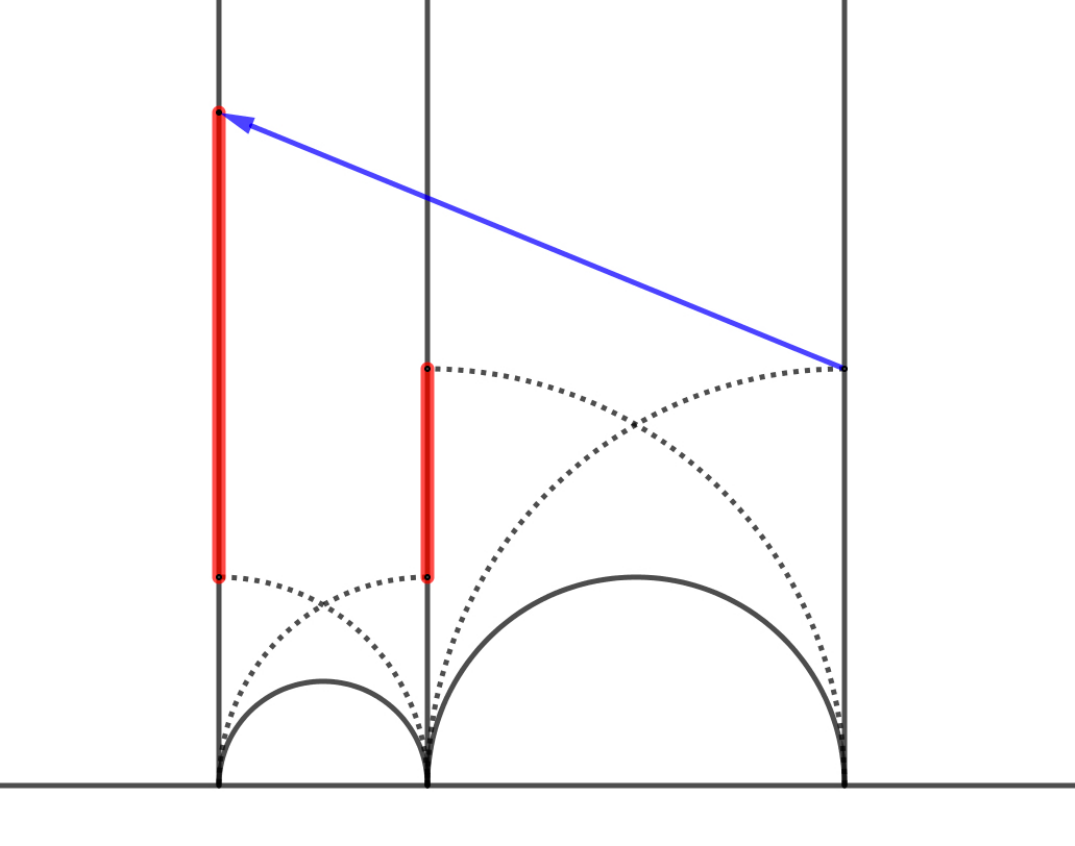}
\caption{Here the shears are shown in red, and have opposite signs. The left and right geodesics are glued by a hyperbolic isometry that takes the basepoint of the blue arrow to its tip.}
\label{F:TwoCusps}
\end{figure}

If you follow the horospherical foliation around a puncture, passing through both triangles, you arrive further out along the edge of the triangle by an amount equal to the sum of the shears, say $|s_1+s_2|$. See Figure \ref{F:TwoCusps}. You can then complete this horospherical path to a loop by traveling this $|s_1+s_2|$ along the geodesic. As you slide this path farther out along the cusps of the triangles, the distance traveled in the horospherical part of this loop goes to zero, so this loop seems to be converging to a geodesic of length $|s_1+s_2|$. 

\begin{lem}
The completion of the surface obtained by gluing together two triangles as above is a pair of pants with boundary geodesics of length $|s_1+s_2|, |s_2+s_3|, |s_3+s_1|$. If any of these three quantities are zero, you instead get a cusp. 
\end{lem}

For a very careful and clear proof, which proceeds using the developing map rather than the informal heuristic we have suggested, see \cite[Section 7.4]{Mar}. 

\begin{figure}[h!]
\includegraphics[width=.35\linewidth]{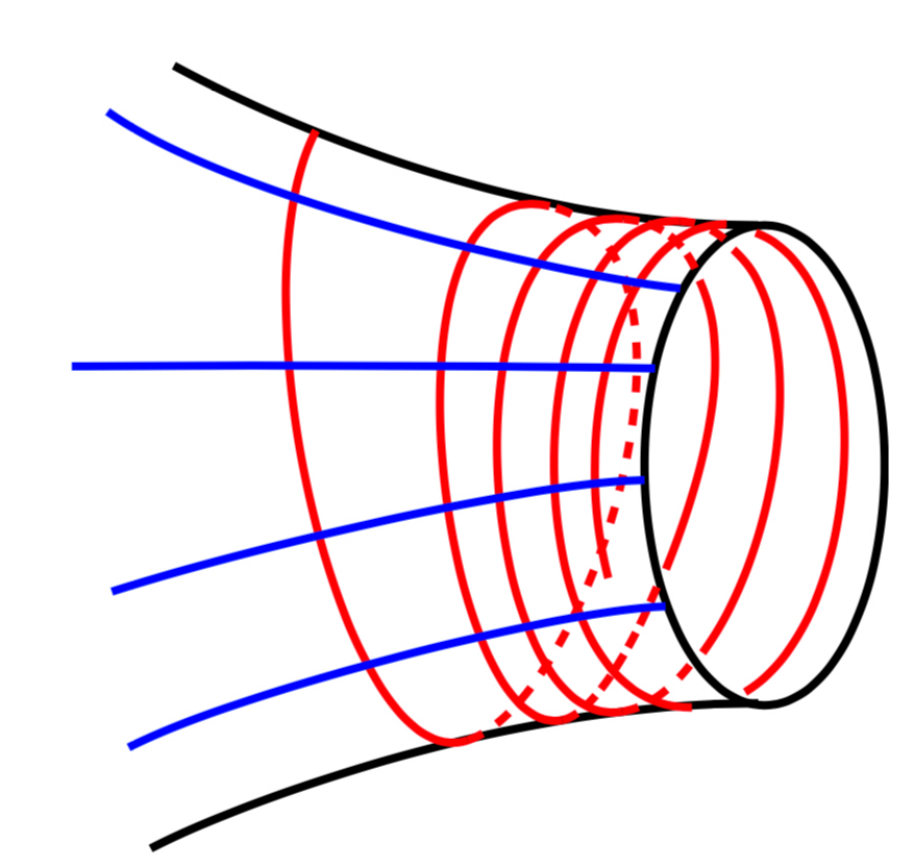}
\caption{Each cusp of each triangle spirals towards one of the three boundary curves of the pants. Image from \cite[Figure 7.20]{Mar}.}
\label{F:Spiral}
\end{figure}

\begin{rem}
As we linearly interpolate between $(s_1, s_2, s_3)$ and $(-s_1, -s_2, -s_3)$, at the halfway point $(0,0,0)$ each boundary component will reach zero length and become a cusp. On one side of the interpolation the triangles will spiral towards the boundary component in one direction, and on the other side they will spiral in the other direction. 
\end{rem}

Return to the situation of a maximal lamination $\lambda$ containing a pants decomposition $\cP$. The $s_i$ above are some of the shear coordinates for $\cT_g$. One also requires shear coordinates for small arcs $A$ passing through a boundary of a pants. This shear coordinate is directly seen to be similar to a Fenchel-Nielsen twist parameter, in that if we do a Fenchel-Nielsen twist by $\e$, the shear changes by $\e$. 

The  Fenchel-Nielsen twist is presumably not the exact same thing as the shear of the arc $A$, even up to a constant.  This is because Fenchel-Nielsen twist parameters are usually defined by considering orthogeodesics from another boundary of the pants. The shear is related to where the central leaves of the horocyclic foliation lands on the cuff. The shear for $A$ should be a function of the Fenchel-Nielsen twist parameter of that curve, and the 5 length parameters for the 2 pants that share this cuff. In this way one can at least see that the map from Fenchel-Nielsen twist parameters to the shear parameters preserves volume, because its derivative can be written as an upper triangular matrix with ones on the diagonal.

\section{Hamiltonian flows}\label{S:Ham}

In fact, both the Thurston and Weil-Petersson volume forms arise from symplectic forms. (Although it is a little tricky to talk about symplectic forms since $\ML$ doesn't have a natural differential structure.) Bonahon-Sozen actually showed that the map from $\cT_g$ to $\ML_\lambda$ is a symplectomorphism. The earthquake flow on $\cT_{g}$ is the Hamiltonian flow for the length of $\lambda$, and the unipotent flow on quadratic differentials with horizontal foliation $\lambda$ is the Hamiltonian flow of the area function. 

Consider a specific $\mu$, and let $\hat{X}$ denote the double cover associated to $q_{\lambda, \mu}$.  We can associate $\lambda$ and $\mu$ to cohomology classes $\hat{\lambda}$ and $\hat{\mu}$, and the area function $A$ is given by 
$$A(\eta)=\langle\hat{\lambda}, \eta\rangle.$$
We now claim that the Hamiltonian vector field is $\hat{\lambda}$. To show this, we compute 
\begin{eqnarray*}
(dA)_\eta(\xi) &=& \left.\frac{d}{dt}\right\vert_{t=0} \langle \hat{\lambda}, \eta+t\xi \rangle \\
&=& \langle \hat{\lambda}, \xi \rangle.
\end{eqnarray*}
This exactly shows that unipotent flow is Hamiltonian. 

\section{The linear structure on $\ML_\alpha$}

If $\alpha$ is maximal, then all foliations  $\mu\in\ML_\alpha$ have singularities in correspondence to the triangular regions of $\alpha$. Hence one can pass consistently to a double cover where $\hat{\mu}$ gives a cohomology class $[\hat{\mu}]$. This maps $\ML_\alpha$ to a vector space. 

\begin{lem}
The map $\mu\mapsto [\hat{\mu}]$ is injective. 
\end{lem}

\begin{proof}
This is equivalent to the statement that if you know the horizontal foliation of an Abelian differential (up to Whitehead moves), and you know the relative cohomology class of the vertical foliation, then you can recover the Abelian differential. (Actually $\hat{\mu}$ lives in the $-1$ eigenspace of the cohomology of the double cover, which is isomorphic to the $-1$ eigenspace of the relative cohomology.) The proof is that knowing the horizontal foliation allows you to determine the IET giving the first return map to any vertical segment, and the cohomology class gives the sizes of the rectangles in an associated zippered rectangle decomposition. See  \cite{MW}. 
\end{proof}

\begin{rem}
This can be interpreted as saying that, passing to the appropriate Teichm\"uller space, a single period coordinate chart suffices for the slice of any stratum where the horizontal foliation is held constant. 
\end{rem}

In fact one can see that the image is a convex polyhedral cone. Typically (ex if $\alpha$ is uniquely ergodic) this cone is a half space. 

Alternatively, one can parameterize  $\ML_\alpha$ by transverse distributions or transverse cocycles, and in this way see that $\ML_\alpha$ has a natural linear structure \cite{B}. Any two points in $\ML_\alpha$ can be joined by a straight line, and the resulting path in $\cT_g$ is called a cataclysm or shear map. It differs from an earthquake in that earthquakes always shear in one direction (right or left), and that earthquakes can be continued for all time, whereas cataclysm paths can cease to be well-defined in finite time. 

\section{Other results on earthquakes}

Thurston proved that, given any two points in $\cT_{g}$, there is a unique earthquake path between them \cite{Thurston2D}. Kerckhoff proved that hyperbolic length functions are convex along earthquake (and even cataclysm paths \cite{Th}).  This was famously used by Kerckhoff to solve the Nielsen realization problem \cite{K}.

\bibliography{earthquakes}{}
\bibliographystyle{amsalpha}
\end{document}